\documentclass[11pt,reqno]{amsart}
\usepackage[utf8]{inputenc}
\usepackage{a4wide}
\usepackage{xparse} 
\usepackage{amssymb,amsthm,amsmath,eucal,mathrsfs,bbm}
\usepackage{enumitem}
\usepackage[colorlinks=true,    
            linkcolor=blue,     
            citecolor=blue,     
            urlcolor=blue]      
           {hyperref}

\def\N{\mathbb{N}}
\def\R{\mathbb{R}}

\def\C{\mathbb{C}}
\def\Z{\mathbb{Z}}

\def\cL{\mathcal{L}}
\def\cT{\mathcal{T}}
\def\cR{\mathcal{R}}
\def\cX{\mathcal{X}}
\def\cY{\mathcal{Y}}

\def\EE{\mathbb{E}}

\def\MP{\emptyset}

\def\loc{{\mathrm{loc}}}

\def\arg{\mathrm{arg}\,}

\def\sm{\setminus}

\def\bel{\begin{equation}}
\def\eel{\end{equation}}

\def\beq{\begin{eqnarray*}}
\def\eeq{\end{eqnarray*}}

\newcommand{\n}[1]{\|#1\|}
\newcommand{\nn}[2]{\|#1\|_{#2}}

\def\eps{\varepsilon}
\def\ph{\varphi}
\def\Om{\Omega}
\def\om{\omega}

\def\la{\lambda}

\def\al{\alpha}
\def\ga{\gamma}
\def\Ga{\Gamma}
\def\si{\sigma}
\def\Si{\Sigma}
\def\del{\delta}

\def\ov{\overline}

\def\vrh{\varrho}

\def\Rep{\mathrm{Re}\,}
\def\Imp{\mathrm{Im}\,}

\makeatletter
\newtheoremstyle{tagged}
  {3pt}
  {3pt}
  {\itshape}
  {}
  {\bfseries}
  {}
  { }
  {\thmname{#1}~(#3)}
\makeatother

\newtheorem{prop}{Proposition}[section]
\newtheorem{lemma}[prop]{Lemma}
\newtheorem{remark}[prop]{Remark}

\newtheorem{definition}[prop]{Definition}
\newtheorem{theorem}[prop]{Theorem}

\theoremstyle{tagged}
\newtheorem{assumption}{Assumption}
\newcommand{\assref}[1]{\upshape (\nameref{#1})}

\renewcommand\labelenumi{$(\roman{enumi})$}
\renewcommand\theenumi\labelenumi

\numberwithin{equation}{section}

\newcounter{aufzi}
\newenvironment{aufzi}{\begin{list}{ {\upshape(\alph{aufzi})}}{
        \usecounter{aufzi}
        \topsep1ex
        \parsep0cm
        \itemsep1ex
        \leftmargin0.8cm
        \labelwidth0.5cm
        \labelsep0.3cm
}}
{\end{list}}

\NewDocumentCommand{\norm}{m o}{%
  \lVert #1 \rVert%
  \IfValueT{#2}{_{#2}}%
}
\NewDocumentCommand{\bignorm}{m o}{%
  \bigl\lVert #1 \bigr\rVert%
  \IfValueT{#2}{_{#2}}%
}
\NewDocumentCommand{\Bignorm}{m o}{%
  \Bigl\lVert #1 \Bigr\rVert%
  \IfValueT{#2}{_{#2}}%
}

  \newcommand {\Bigabs}[1]{\Bigl| #1 \Bigr|}

  \newcommand {\idual}[3] {\langle #1, #2 \rangle_{#3}}

  \newcommand {\dual}[2] {\idual{#1}{#2}{} }
  \newcommand {\sect}[1] {\Sigma_{#1}}
  
  \newcommand {\eins}{\mathbbm 1}
  \newcommand {\sfrac}[2] { {\,{}^{#1}\!\!/\!{}_{#2}\,}} 
  \newcommand {\einhalb} {\sfrac{1}{2}}

\title{A cheap way to closed operator sums}

\author{Benhard H. Haak}
\address{\it Université de Bordeaux\\Institut de Mathématiques de Bordeaux\\351 cours de la Libération\\F -- 33405 Talence\\France}
\email{bernhard.haak@math.u-bordeaux.fr}

\author{Peer Chr. Kunstmann}
\address{\it Karlsruhe Institute of Technology (KIT),
Institute for Analysis\\
Englerstr. 2, D -- 76128 Karlsruhe\\Germany}
\email{peer.kunstmann@kit.edu}

\date{}

\begin{document}

\dedicatory{Dedicated to Robert Denk on the occasion of his 60th birthday.}

\begin{abstract}
  Let $A$ and $B$ be sectorial operators in a Banach space $X$ of
  angles $\om_A$ and $\om_B$, respectively, where
  $\om_A+\om_B<\pi$. We present a simple and common approach to
  results on closedness of the operator sum $A+B$, based on
  Littlewood-Paley type norms and tools from several interpolation
  theories.  This allows us to give short proofs for the well-known
  results due to Da~Prato-Grisvard and Kalton-Weis. We prove a new
  result in $\ell^q$-interpolation spaces and illustrate it with a
  maximal regularity result for abstract parabolic equations. Our
  approach also yields a new proof for the Dore-Venni result.


\end{abstract}

\maketitle

\allowdisplaybreaks

\section{Introduction and main idea of the approach}\label{sec:intro}

In this paper we consider the problem of closedness in a Banach space
$X$ of an operator sum of two closed operators $A$ and $B$ with
non-empty resolvent sets whose resolvent operators commute. More
specifically we assume that $A$ and $B$ are sectorial operators in $X$
of angle $\om_A\in[0,\pi)$ and $\om_B\in[0,\pi)$, respectively, with
dense domain and range and that $\om_A+\om_B<\pi$.
We recall that sectorial operators with dense range are injective.

The sum $A+B$ has domain $D(A+B)=D(A)\cap D(B)$, which is a Banach space for
the norm $x\mapsto \n{x}+\n{Ax}+\n{Bx}$, and closedness of $A+B$ is -- via
the open mapping theorem -- equivalent to an estimate
\begin{align}\label{eq:equiv}
 \n{x}+\n{Ax}+\n{Bx}\le C \big(\n{x}+\n{(A+B)x}\big),\quad x\in D(A)\cap D(B). 
\end{align}
A prominent example of such a situation is frequently encountered in
the study of parabolic evolution equations
\begin{align}\label{eq:maxregT}
 u'(t)+A\, u(t)=f(t),\quad t\in(0,T),\qquad u(0)=0,
\end{align}
with $0<T\le\infty$, if one is interested in maximal regularity in the sense
that $f\in L^p(0,T;X)$, say, should imply $u',Au\in L^p(0,T;X)$
as well.  Here $-A$ is the generator of a bounded analytic semigroup
in $X$, hence a sectorial operator of angle $\om_A<\frac\pi2$.  In
order to view this as a problem on operator sums we have to consider
a lifted version $u(\cdot)\mapsto Au(\cdot)$ of $A$ in the space
$L^p(0,T;X)$ with domain $L^p(0,T;D(A))$ where $D(A)$ is
equipped with the graph norm. By abuse of notation we denote this
lifting still by $A$.  The operator $B$ is given by $\frac{d}{dt}$
with domain $\{u\in W^{1,p}(0,T;X): u(0)=0 \}$ in
$L^p(0,T;X)$.  It is clear that $B$ is sectorial of angle
$\om_B=\frac\pi2$ and that resolvents of $A$ and $B$ commute.

For finite $T<\infty$, the mild solution $u$ of \eqref{eq:maxregT} for a
given $f\in L^p(0,T;X)$ satisfies
$\norm{u}_{L^p(0,T;X)} \lesssim \norm{f}_{L^p(0,T;X)}$ and the problem
of maximal regularity then reduces to \eqref{eq:equiv}, i.e., to the
question whether the operator $A+B$ with domain $D(A)\cap D(B)$ is
closed in $L^p(0,T;X)$.  For $T=\infty$, the situation changes a bit
since, for $f\in L^p(0,\infty;X)$ we then only have
$u\in L^p_{\loc}(0,\infty;X)$ for the mild solution to
\eqref{eq:maxregT}. Here, the question whether
$u',Au\in L^p(0,\infty;X)$ reduces to the existence of a constant
$C>0$ with
\begin{align}\label{eq:equiv-hom}
 \n{Au}+\n{Bu}\le C \n{(A+B)u},\quad u\in D(A)\cap D(B). 
\end{align}
This is stronger than \eqref{eq:equiv} and, if $A+B$ is injective, it is
equivalent to boundedness of (an extension of) $A(A+B)^{-1}$ or $B(A+B)^{-1}$.
Boundedness of one of these operators also entails \eqref{eq:equiv}, hence
closedness of $A+B$, and this is the usual way to study closedness of operator sums, which we also follow in this paper.

\noindent In the maximal regularity problem above, $L^p(0,T;X)$ is just an
example. The question is also relevant for various other $X$-valued
function spaces on $(0,T)$ or on $\R_+$. The main motivation for the
study of maximal regularity for linear problems is to be able to do
fixed point arguments without loss of regularity for, e.g. quasilinear
parabolic equations. For maximal $L^p$-regularity for linear problems
of elliptic and parabolic type we refer to the seminal work
\cite{Denk-Hieber-Pruess}, and e.g. to
\cite{ClementLi,Denk:MR-parabolic} for applications of maximal
regularity to quasilinear problems.

\noindent Operator sum theorems of different types have been given in the
literature over the past 50 years and we mention as most influential for
several variants the results by Da~Prato-Grisvard (\cite{DaP-G},
1975), by Dore-Venni (\cite{Do-V}, 1987), and by Kalton-Weis
(\cite{Ka-W}, 2001), and each had immediate impact on the study of
parabolic evolution equations.  Here we present a simple common
approach to such operator sum results and prove a new one. Our
approach can, for the result by Da~Prato-Grisvard, be sketched as
follows.

\noindent For simplicity, we first assume that $-A$ generates a
bounded analytic semigroup $(T(t))_{t\ge0}$ and that $-B$ generates a
bounded semigroup $(S(t))_{t\ge0}$.  If these semigroups commute, then
$(T(t)S(t))_{t\ge0}$ is a strongly continuous semigroup whose
generator is the closure $\ov{A+B}$ of the operator $A+B$.  In order
to see this, observe that $A+B$ has domain $D(A+B)=D(A)\cap D(B)$
which is dense and invariant under the semigroup operators $T(t)S(t)$
and apply a standard result from semigroup theory (see, e.g., \cite[Prop.~II.1.7]{Engel-Nagel}).

\noindent For suitable $x\in X$ we then have
\begin{equation}\label{eq:StTt}
 Sx:=\big(\ov{A+B}\big)^{-1} x =\int_0^\infty S(t)T(t)x\,dt.
\end{equation}
We basically have to show that $AS$ (or $BS$) extends to a bounded
operator on $X$, see Section~\ref{sec:prelim}. In order to do so we
apply suitable functionals $x'\in X'$ and write
\[
  \langle ASx, x'\rangle=\int_0^\infty \langle A S(t)T(t)x, x'\rangle\,dt 
 = \int_0^\infty \langle S(t)(tA)^{\einhalb}T(\sfrac{t}{2})x,((tA)T(\sfrac{t}{2}))'x'\rangle\,\frac{dt}t,
\]
using that $S(t)$ and $T(t)$ commute.  The key assumption for the Da
Prato-Grisvard result is that $X$ is a real interpolation space or, in
other words, that the norm of $X$ is equivalent to a certain
Littlewood-Paley norm, namely
\begin{align}\label{eq:normXeqsim}
 \nn{x}{X}\simeq \nn{t\mapsto (tA)^{\einhalb}T(\sfrac{t}{2})x}{L^p(0,\infty;\frac{dt}t,X)}
\end{align}
for some $p\in[1,\infty]$. This in turn is equivalent to the two
estimates
\begin{align*}  
 \nn{t\mapsto (tA)^{\einhalb}T(\sfrac{t}{2})x}{L^p(0,\infty;\frac{dt}t,X)}
  \lesssim  &  \; \nn{x}{X}, \\
 \nn{t\mapsto ((tA)^{\einhalb}T(\sfrac{t}{2}))'x'}{L^{p'}(0,\infty;\frac{dt}t,X')}
  \lesssim  & \; \nn{x'}{X'}.
\end{align*}
Since
\begin{align}\label{eq:B-bdd}
  \nn{t\mapsto S(t)(tA)^{\einhalb}T(\sfrac{t}{2})x}{L^p(0,\infty;\frac{dt}t,X)}\le\left(\sup_t\n{S(t)}\right)
   \nn{t\mapsto (tA)^{\einhalb}T(\sfrac{t}{2})x}{L^p(0,\infty;\frac{dt}t,X)}
\end{align}
we can conclude via H\"older
\begin{align*}
  |\langle ASx, x'\rangle|
  \le &  \; \nn{t\mapsto S(t)(tA)^{\einhalb}T(\sfrac{t}{2})x}{L^p(0,\infty;\frac{dt}t,X)}
  \nn{t\mapsto ((tA)^{\einhalb}T(\sfrac{t}{2}))'x'}{L^{p'}(0,\infty;\frac{dt}t,X')} \\
  \lesssim & \; \left(\sup_t\n{S(t)}\right)\,\nn{x}{X}\nn{x'}{X'},
\end{align*}
and use that $(S(t))$ is bounded. Clearly, this approach should also
work for other Littlewood-Paley norms. In the general case, however,
we cannot use \eqref{eq:StTt} and we resort to a formula involving
resolvents, namely
\begin{align}\label{eq:def-Sintro}
   Sx=\frac1{2\pi i}\int_\Ga (\la+B)^{-1}R(\la,A)x\,d\la,
\end{align}
where $\Ga$ is the boundary of a sector $\sect{\rho}$ with
$\rho\in(\om_A,\pi-\om_B)$ oriented such that $\sect{\om_A}$ is to the
left of $\Ga$.

We show that this approach yields a unified perspective for the known
results on closedness of operator sums by Da~Prato-Grisvard
(\cite{DaP-G}), clearly tied to the real interpolation method, and by
Kalton-Weis (\cite{Ka-W}). The latter uses a bounded
$H^\infty$-calculus for one operator. Boundedness of the
$H^\infty$-calculus, however, can be characterized (under a mild
assumption on the Banach space) by estimates for square function norms
involving Littlewood-Paley expressions for this operator. Here, we
will work with these Littlewood-Paley norms directly.

The proofs we give for those results on the closedness of operator
sums are ``cheap'' in the sense that we just use norms on
Littlewood-Paley type expressions for one operator (and its dual) and
a corresponding sectoriality estimate for the other operator. Similar
to the above sketch, the proof itself is then done in a few lines. We
demonstrate this in Section~\ref{sec:DaP-G} for the Da~Prato-Grisvard
result in real interpolation spaces and in Section~\ref{sec:Ka-W} for
the Kalton-Weis result.

In Section~\ref{sec:Tr-L} we establish a new result on closedness of
operator sums in the generalized Triebel-Lizorkin spaces that were 
introduced in \cite{KUll}. This can be
viewed as a counterpart to the Da~Prato-Grisvard result where real
interpolation is replaced by the $l^q$-interpolation method that was 
introduced in \cite{Ku:lq}. Since we also need to know something 
about the dual situation, as can be seen from the sketch above, and 
the dual situation is not covered in \cite{KUll, Ku:lq} we provide the 
necessary material in Appendix~\ref{app:dual-lq}.

It came as a bit of a surprise to us that also the Dore-Venni result
fits into our setting. In order to see this we had to reconsider the
complex interpolation method from our perspective in the sense that we
need a ``real type'' formulation. We do this in
Appendix~\ref{app:complex}, the principal inspiration coming from
\cite{LL-interpol}. The approach itself can be seen as a variant of
certain arguments used for other but related purposes in
\cite{Pr-So}. We present our proof of the Dore-Venni result in
Section~\ref{sec:Do-V}. Maybe it is less convincing to term this a ``cheap
way'' but it makes very transparent that the technicalities we face
are due to intrinsic properties of the complex interpolation method.

\section{Preliminaries}\label{sec:prelim}

Let $A$ and $B$ be closed and injective linear operators in a Banach
space $X$ and assume $\rho(A)\cap\rho(B)\neq\MP$ and that resolvents
of $A$ and $B$ commute. Further assume that $D(A)\cap D(B)$ is dense
in $X$. We consider the operator $A+B$ with $D(A+B)=D(A)\cap D(B)$.

We present two lemmas that reduce the proof of closedness of the
operator $A+B$ to the verification of certain properties for an operator
$S$. These properties will then ensure that the closure $\ov{S}$ of $S$ equals $(A+B)^{-1}$.

\begin{lemma}\label{lem:1}
  Let $S$ be a densely defined linear operator in $X$ with range
  $R(S)\subseteq D(A)\cap D(B)$. Assume that $(A+B)Sx=x$ for all
  $x \in D(S)$ and that one of the operators $AS$, $BS$ has a bounded extension, i.e.
  $\ov{AS}\in\cL(X)$ or $\ov{BS}\in\cL(X)$.  Then $S$ is closable, the closure
  $\ov{S}$ is injective and maps $D(\ov{S})$ into $D(A)\cap D(B)$ with
  $(A+B)\ov{S}x=x$ for all $x\in D(\ov{S})$.
\end{lemma}
\begin{proof}
  First note that $\ov{AS},\ov{BS}\in\cL(X)$ by $(A+B)S\subseteq I_X$ and denseness of $D(S)$.
  
  \smallskip\noindent
  Let $(x_n)$ be a sequence in $D(S)$ and $z\in X$ such that
  $x_n\to 0$ and $Sx_n\to z$. Then $ASx_n=\ov{AS}x_n\to 0$ and
  $BSx_n=\ov{BS}x_n\to0$. Closedness of and $A$ and $B$ implies
  $z\in D(A)\cap D(B)$ and $Az=0=Bz$. By injectivity of $A$ or $B$ we
  conclude that $z=0$. Hence $S$ is closable.
  
  \smallskip\noindent 
  Now let $(x_n)$ be a sequence in $D(S)$ and $x,z\in X$ with
  $x_n\to x$ and $Sx_n\to z$. Then $ASx_n=\ov{AS}x_n\to \ov{AS}x$ and
  $BSx_n=\ov{BS}x_n\to \ov{BS}x$. By closedness of $A$ and $B$ we
  conclude $z\in D(A)\cap D(B)$ and $Az=\ov{AS}x$,
  $Bz=\ov{BS}x$. Hence $(A+B)z=(\ov{AS}+\ov{BS})x=x$. In particular,
  if $z=0$ then $x=0$ so $\ov{S}$ is injective.  Moreover, the
  argument shows $\ov{S}:D(\ov{S})\to D(A)\cap D(B)$ and
  $(A+B)\ov{S}x=x$ for all $x\in D(\ov{S})$, i.e.
  $(A+B)\ov{S}\subseteq I_X$.
\end{proof}

\begin{lemma}\label{lem:2}
  In addition to the assumptions of Lemma~\ref{lem:1}, assume that the
  linear operator $S$ commutes with resolvents of $A$ and of $B$.  If
  $\ov{S}$ is bounded or if there exists $\la\in\rho(A)\cap\rho(-B)$
  such that $R(\la,A)(\la+B)^{-1}D(\ov{S})$ is a core for $A+B$, then
  $A+B$ is closed and injective and $\ov{S}=(A+B)^{-1}$.
\end{lemma}

\begin{proof}
  From the assumptions we obtain that also the operator $\ov{S}$
  commutes with the resolvents of $A$ and $B$.
  
  \smallskip\noindent 
  If $\ov{S}\in\cL(X)$, we conclude that $\ov{S}$ commutes with $A$
  and $B$.  Let $(x_n)$ be a sequence in $D(A+B)$ and $x,z\in X$ with
  $x_n\to x$ and $(A+B)x_n\to z$.  We then have $\ov{S}x_n\to \ov{S}x$
  and
  \[
    (A+B)\ov{S}x_n=\ov{S}(A+B)x_n\to \ov{S}z.
  \]
  Furthermore, $A\ov{S}x_n=\ov{AS}x_n\to\ov{AS}x$ and
  $B\ov{S}x_n=\ov{BS}x_n\to\ov{BS}x$, hence
  $(A+B)\ov{S}x_n\to (\ov{AS}+\ov{BS})x=x$. Thus we have
  $x=\ov{S}z$. In particular, $x\in R(\ov{S})\subseteq D(A+B)$ and
  $(A+B)x=(A+B)\ov{S}z=z$. We have shown that $A+B$ is closed. We also
  see that $z=0$ implies $x=0$, and can conclude that $A+B$ is
  injective and $(A+B)^{-1}=\ov{S}$.
  
  \smallskip\noindent
  Now let $\la\in\rho(A)\cap\rho(-B)$ such that
  $D:=R(\la,A)(\la+B)^{-1}D(\ov{S})$ is a core for $A+B$. We calculate
  \begin{align*}
    (A+B)R(\la,A)(\la+B)^{-1}
    & = \; AR(\la,A)(\la+B)^{-1}+B(\la+B)^{-1}R(\la,A)   \\
    & = \; \big(\la R(\la,A)-I\big)(\la+B)^{-1}+\big(I-\la(\la+B)^{-1}\big)R(\la,A)    \\
    & = \; R(\la,A)-(\la+B)^{-1}.
\end{align*} 
Hence $(A+B)R(\la,A)(\la+B)^{-1}$ commutes with $\ov{S}$, i.e.
\begin{align*}
 (A+B)R(\la,A)(\la+B)^{-1}\ov{S}\subseteq \ov{S}(A+B)R(\la,A)(\la+B)^{-1},
\end{align*}
where the operator on the left has domain $D(\ov{S})$. On the other
hand, we have, by the commuting assumption,
$R(\ov{S})\subseteq D(A+B)$, and Lemma~\ref{lem:1},
\begin{align*}
 (A+B)R(\la,A)(\la+B)^{-1}\ov{S}\subseteq (A+B)\ov{S}R(\la,A)(\la+B)^{-1}\subseteq R(\la,A)(\la+B)^{-1}.
\end{align*}
We conclude
\begin{equation}\label{eq:D}
  \ov{S}(A+B)y=y \quad \mbox{ for all $y\in R(\la,A)(\la+B)^{-1}(D(\ov{S}))=D$}.
\end{equation}
Now let $x\in D(A+B)$. By assumption we find a sequence $(y_n)$ in
$D(\ov{S})$ such that, for $(x_n):=(R(\la,A)(\la+B)^{-1}y_n)$, we have
$x_n\to x$ and $(A+B)x_n\to (A+B)x$.  By $x_n\in D$ and \eqref{eq:D}
we have $\ov{S}(A+B)x_n=x_n\to x$. Since $\ov{S}$ is closed we obtain
$(A+B)x\in D(\ov{S})$ and $\ov{S}(A+B)x=x$.  The assertions follow.
\end{proof}

\begin{remark}\rm
  We draw attention to the fact that, in the proof of
  Lemma~\ref{lem:1}, it is sufficient to assume that $A$ \emph{or} $B$
  is injective. We also comment that boundedness of $\ov{S}$ in
  Lemma~\ref{lem:2} implies that $R(\la,A)(\la+B)^{-1}D(\ov{S})$ is a
  core for $A+B$ for any $\la\in\rho(A)\cap\rho(-B)$. However, if
  $\ov{S}$ is bounded we do not need the condition
  $\rho(A)\cap\rho(-B)\neq\MP$.
\end{remark}

The following is the main assumption on the operators $A$ and $B$ in this paper.

\begin{assumption}[$\Sigma$]\rm\label{ass:main}
  We assume that $A$ and $B$ are sectorial operators in $X$ with dense
  domain and range and with sectoriality angles $\om_A$ and $\om_B$,
  respectively. Moreover, we assume that resolvents of $A$ and $B$
  commute and that $\om_A+\om_B<\pi$ and let
  $\rho\in(\om_A,\pi-\om_B)$.
\end{assumption}

Observe that the rays $e^{\pm i\rho}(0,\infty)$ are contained in $\rho(A)\cap\rho(-B)$. Before we define the operator $S$ we recall the following.

\begin{lemma}\label{lem:Ainv}
Under Assumption~\assref{ass:main},
the operators $A$ and $B$ are injective and $A^{-1}$ and $B^{-1}$ are sectorial of angle $\om_A$ and $\om_B$, respectively, and with dense domain and range. 
\end{lemma}
\begin{proof}
Injectivity of a sectorial operator follows from denseness of its range. The formula
\begin{align}\label{eq:A-1sect}
 \la(\la-A^{-1})^{-1}=A(A-\la^{-1})^{-1},\qquad \la\in\rho(A)\sm\{0\}.
\end{align}
implies sectoriality of $A^{-1}$.
\end{proof}

The following lemma sheds some light on spaces we shall use.
 
\begin{lemma}\label{lem:prod-res}
Under Assumption~\assref{ass:main}, let $k,l,m,n\in\N_0$. Then the range of the operator
\begin{align}\label{eq:range-prod-res}
 \left(\prod_{j=1}^k R(\mu_j,A)\right) \left(\prod_{j=1}^l R(\nu_j,A^{-1})\right) \left(\prod_{j=1}^m (\al_j+B)^{-1}\right)
 \left(\prod_{j=1}^k (\beta_j+B^{-1})^{-1}\right) 
\end{align}
is invariant under resolvents of $A,B,A^{-1},B^{-1}$ and independent of $\mu_j,\nu_j\al_j,\beta_j\in e^{\pm i\rho}(0,\infty)$.
\end{lemma}
\begin{proof}
The stated invariance under resolvents is clear from the assumptions. Independence can be proved by induction on $k+l+m+n$ via the resolvent equation.
\end{proof}

\begin{definition}\rm\label{def:D_}
Under the Assumption~\assref{ass:main} and for $k,l,m,n\in\N_0$, we denote the range of the operator in \eqref{eq:range-prod-res} by $D_{k,l,m,n}$. 
\end{definition}

\begin{remark}\rm\label{rem:D_}
Then we have, e.g. $D_{k,l,0,0}=D(A^k)\cap D(A^{-l})=D(A^k)\cap R(A^l)$ and $D_{0,0,m,n}=D(B^m)\cap D(B^{-n})=D(B^m)\cap R(B^n)$. It can be shown that
\begin{align}\label{eq:D1010}
 D_{1,0,1,0}=D(BA)\cap D(B)=D(AB)\cap D(A). 
\end{align}
Indeed, for the first equality in \eqref{eq:D1010} the inclusion ``$\subseteq$'' is clear. For the reverse inclusion let $x\in D(BA)\cap D(B)$ and put $y:=(\mu+B)x$, $z:=(\mu+B)Ax$ where $|\arg\mu|=\rho$. Then $(\mu-A)x=(\mu+B)^{-1}(\mu y-z)$, hence $x=R(\mu,A)(\mu+B)^{-1}(\mu y-z)\in D_{1,0,1,0}$. The second equality in \eqref{eq:D1010} then holds by interchanging the roles of $A$ and $B$.
\end{remark}

\begin{lemma}\label{lem:D_}
The spaces $D_{k,l,m,n}$ are dense in $X$. Equipping $D_{k,l,m,n}$ with the norm that turns one of the operators in \eqref{eq:range-prod-res} into an isometry $X\to D_{k,l,m,n}$, we have that, for $\tilde{k}\ge k$, $\tilde{l}\ge l$, $\tilde{m}\ge m$, $\tilde{n}\ge n$, the space $D_{\tilde{k},\tilde{l},\tilde{m},\tilde{n}}$ is dense in $D_{k,l,m,n}$. 
\end{lemma}
\begin{proof}
The second assertion follows from the first, so we prove the first assertion. For 
\begin{align}\label{eq:def-rho-n}
\vrh_n(\la)=\frac{n}{n+\la}-\frac1{1+n\la},\quad \la\in\C\sm(-\infty,0], n\in\N,
\end{align} 
we let 
\begin{align}\label{eq:Un}
 U_n:=\vrh_n(e^{i(\pi-\rho)}A)\vrh_n(e^{-i\rho}B). 
\end{align}
Observe that $e^{i(\pi-\rho)}A$ and $e^{-i\rho}B$ are sectorial by choice of $\rho$.
Let $k_0:=\max\{k,l,m,n\}$. Then, for any $x\in X$, we have $U_n^{k_0}x\in D_{k,l,m,n}$, and $U_n^{k_0}x\to x$ in $X$ holds by \cite[Prop.~9.4]{KuW:levico}.
\end{proof}

\section{The formula for $S$ and the key lemma}\label{sec:key}

\noindent We define the operator $S$ as in \eqref{eq:def-Sintro}.

\begin{definition}\rm\label{def:S-key} 
  Under Assumption~\assref{ass:main} we define
\begin{align}\label{eq:defS-key}
 Sx:=\frac1{2\pi i}\int_\Ga (\la+B)^{-1}R(\la,A)x\,d\la,\quad x\in D(S):=D_{1,1,1,1},
\end{align}
where $\Ga$ is the boundary of the sector $\sect{\rho}$ oriented such that $\sect{\om_A}$ is to the left of $\Ga$. 
\end{definition}

\noindent The integral in \eqref{eq:defS-key} is abolutely convergent for $x\in D(S)$ since
\begin{align}\label{eq:DS}
 D(S)=D_{1,1,1,1}\subseteq  R(AB)\cap R(BA)\subseteq R(A)\cap R(B).
\end{align} 
The next lemma gathers the properties of $S$ we need.

\begin{lemma}\label{lem:propS}
Under Assumption~\assref{ass:main} the operator $S$ defined in \eqref{eq:defS-key} satisfies the following:
\begin{enumerate}
\item \label{it:L3.6i}$S$ commutes with resolvents of $A$ and of $B$.
\item \label{it:L3.6ii} For $\la\in\Ga$, $(\la+B)^{-1}R(\la,A)(D(S))$ is a core for $A+B$.
\item \label{it:L3.6iii}
For $\si\in(\rho,\pi)$, $\ph\in H^\infty_0(\sect{\si})$ and $j=0,1$, we have
\[
 \ph(A)B^jSx=\frac1{2\pi i}\int_\Ga \ph(\la) B^j(\la+B)^{-1}R(\la,A)x\,d\la,\quad x\in D(S).
\]
\item \label{it:L3.6iv} For $x\in D(S)$, we have $Sx\in D(A)\cap D(B)$ and $(A+B)Sx=x$.
\end{enumerate}
\end{lemma}
\begin{proof}
\ref{it:L3.6i} is clear. 

\ref{it:L3.6ii} By Lemma~\ref{lem:prod-res} we have, for $\la\in\Ga$ and in the notation of Remark~\ref{rem:D_}, by \eqref{eq:A-1sect},
\[
 (\la+B)^{-1}R(\la,A)(D(S))=D_{2,1,2,1}.
\] 
This set is a core for $A+B$, since it is dense in $D(A+B)=D(A)\cap D(B)$ for the norm $\nn{x}{X}+\nn{Ax}{X}+\nn{Bx}{X}$. 
The proof is similar to the proof of Lemma~\ref{lem:D_}. For $U_n$ as in \eqref{eq:Un} and $x\in D(A+B)$ we have 
$U_n^2x\in D_{2,2,2,2}\subseteq D_{2,1,2,1}$ and $U_nx\to x$, $AU_nx=U_nAx\to Ax$, $BU_nx=U_nBx\to x$ in $X$.
In particular, $D_{2,1,2,1}$ is a core for $A+B$.

\ref{it:L3.6iii} Clearly, $\ph(A)$ commutes with $B^jS$. We choose $\rho'\in(\rho,\si)\cap(\om_A,\pi-\om_B)$, 
$\Ga':=\partial\sect{\rho'}$ and then have
\begin{align*}
B^jS\ph(A)x=&\frac1{2\pi i} \int_\Ga B^j(\la+B)^{-1}R(\la,A) \frac1{2\pi i}\int_{\Ga'} \ph(\mu) R(\mu,A)x\,d\mu\,d\la.
\end{align*}
As usual we use the resolvent equation $R(\la,A)R(\mu,A)=\frac1{\mu-\la}(R(\la,A)-R(\mu,A))$ and observe
\begin{align*}
 \frac1{2\pi i} \int_\Ga B^j(\la+B)^{-1} \frac1{2\pi i}\int_{\Ga'}\frac{\ph(\mu)}{\mu-\la}\,d\mu\, R(\la,A)x\,d\la
 =\frac1{2\pi i}\int_\Ga \ph(\la) B^j(\la+B)^{-1}R(\la,A)x\,d\la
\end{align*}
by Cauchy's integral theorem, since $\la$ is inside of $\Ga$. For the other term we observe that $\la\mapsto \frac1{\mu-\la} B^j(\la+B)^{-1}R(\mu,A)x$ is analytic in $\sect{\rho'}$. By $x\in D(S)$ the integral $\int_\Ga$ is absolutely convergent, and thus vanishes by Cauchy's theorem.

\ref{it:L3.6iv} We have $R(S)\subseteq D(A)\cap D(B)$: For $x\in D_{1,1,1,1}$ we have $Ax,Bx\in D_{0,1,0,1}$ and the integral in \eqref{eq:defS-key} with $Ax$ or $Bx$ in place of $x$ is still absolutely convergent. By the commuting assumption and Hille's theorem we conclude that $Sx\in D(A)\cap D(B)$.

We have to show $(A+B)Sx=x$. Let $\ph(z)=z(1+z)^{-3}$. Then $\ph$ and $z\mapsto z\ph(z)$ belong to $H^\infty_0(\sect{\si})$ for any $\si\in(\rho,\pi)$ and $\ph(A)=A^2(1+A)^{-3}$ is injective. 
Hence it suffices to show $\ph(A)(A+B)Sx=\ph(A)x$. By \ref{it:L3.6iii} we have
\begin{align*}
 \ph(A)(A+B)Sx=&\ \frac1{2\pi i}\int_\Ga \ph(\la)(\la+B)(\la+B)^{-1}R(\la,A)x\,d\la \\
        =&\ \frac1{2\pi i}\int_\Ga \ph(\la)R(\la,A)x\,d\la=\ph(A)x,
\end{align*}
as desired.
\end{proof}

By Lemma~\ref{lem:propS}, we can apply Lemmas~\ref{lem:1} and \ref{lem:2} to the operator $S$ and obtain closedness of $A+B$ by showing that $AS$ has a bounded extension. We need one more formula before we can start the preparation of the corresponding estimates.

\begin{lemma}\label{lem:likeKW}
Under Assumption~\assref{ass:main} we have, for the operator $S$ defined in \eqref{eq:defS-key},
\begin{align*}
 ASx=\frac1{2\pi i}\int_\Ga  \la^{\einhalb} (\la+B)^{-1} A^{\einhalb}R(\la,A)x\,d\la, \quad x\in D(S).
\end{align*}
\end{lemma} 
\begin{proof}
We suitably modify the argument of \cite[Prop. 4.2]{Ka-W} and first take $x=\vrh_1(A)y$ with $\vrh_n$ from \eqref{eq:def-rho-n}. By Lemma~\ref{lem:propS}~\ref{it:L3.6iii} we then have, for $s\in(0,1)$ and $n\in\N$,
\begin{align*}
\vrh_n(A)ASx=&\vrh_n(A)A^{1-s}SA^s\vrh_1(A)y=\frac1{2\pi i}\int_\Ga \vrh_n(\la) \la^{1-s} (\la+B)^{-1} A^{s}R(\la,A)\vrh_1(A)y\,d\la \\
 =& \frac1{2\pi i}\int_\Ga \vrh_n(\la) \la^{1-s} (\la+B)^{-1} A^{s}R(\la,A)x\,d\la.
\end{align*}
Now we let $n\to\infty$ using dominated convergence for the integral. Then we take $s=\frac12$.
\end{proof}  

The key lemma of our approach to closedness of operator sums is the following.

\begin{lemma}\label{lem:key}
Let Assumption~\assref{ass:main} hold and let the operator $S$ be defined by \eqref{eq:defS-key}. If 
\begin{align}\label{eq:key-lem-est}
 \left|\int_0^\infty \langle \psi_\pm(tB) \ph_\pm(tA)x, \big( \ph_\pm(tA)\big)'x' \rangle\,\frac{dt}t \right|
 \lesssim\nn{x}{X}\nn{x'}{X'},
\end{align}
where 
\begin{align}\label{eq:def-psi-ph}
 \psi_\pm(z)=\frac1{e^{\pm i\rho}+z}\quad\mbox{and}\quad
 \ph_\pm(z)=\frac{z^{1/4}}{(-e^{\pm i\rho}+z)^{\einhalb}},
\end{align}
then $AS$ extends to a bounded operator on $X$ and $A+B$ with $D(A+B)=D(A)\cap D(B)$ is closed.
\end{lemma}

\begin{proof}
We take suitable $x'\in X'$ and have, by Lemma~\ref{lem:likeKW},
\begin{align*}
 \langle ASx,x'\rangle &= \frac1{2\pi i} \int_\Ga \langle \la(\la+B)^{-1} \la^{\einhalb}A^{\einhalb}R(\la,A)x,x' \rangle\,\frac{d\la}\la \\
 &= - \frac1{2\pi i} \int_\Ga \langle \la(\la+B)^{-1} \la^{1/4}A^{1/4}(-\la+A)^{-\einhalb}x, 
                                                \la^{1/4}\big(A^{1/4}(-\la+A)^{-\einhalb}\big)'x' \rangle\,\frac{d\la}\la.
\end{align*}
We take the modulus and split $\int_\Ga$ into the two parts
$\int_{\Ga^\pm}$ in directions of $e^{\pm i\rho}$. 

We parametrize $\Ga^\pm$ by $\la=te^{\pm i\rho}$ where $t$ runs from
$\infty$ to $0$ or from $0$ to $\infty$. We thus obtain, up to a constant of modulus $1$,
\begin{align*}
 & \frac1{2\pi} \int_0^\infty \langle t(e^{\pm i\rho}t+B)^{-1} (tA)^{1/4}(-e^{\pm i\rho}t+A)^{-\einhalb}x, 
                                                \big( (tA)^{1/4}(-e^{\pm i\rho}t+A)^{-\einhalb}\big)'x' \rangle\,\frac{dt}t \\
 = \,\,& \frac1{2\pi} \int_0^\infty \langle (e^{\pm i\rho}+tB)^{-1} (tA)^{1/4}(-e^{\pm i\rho}+tA)^{-\einhalb}x, 
                                                \big( (tA)^{1/4}(-e^{\pm i\rho}+tA)^{-\einhalb}\big)'x' \rangle\,\frac{dt}t  \\
 =\,\, & \frac1{2\pi} \int_0^\infty \langle \psi_\pm(tB) \ph_\pm(tA)x, 
                                                \big( \ph_\pm(tA)\big)'x' \rangle\,\frac{dt}t.                                               
\end{align*}
Hence, if \eqref{eq:key-lem-est} holds then $AS$ extends to a bounded operator on $X$. 
As said before, Lemma~\ref{lem:propS}, allows to apply Lemmas~\ref{lem:1} and \ref{lem:2} to the operator $S$, 
and we thus obtain the claim.
\end{proof}  
  
\section{The Da~Prato-Grisvard result}\label{sec:DaP-G}

\begin{theorem}[Da~Prato-Grisvard]
Let Assumption~\assref{ass:main} hold, let $p\in[1,\infty]$, and assume 
\begin{align}\label{eq:DP-sim}
  \nn{t\mapsto\phi(tA)x}{L^p(0,\infty;\frac{dt}t,X)} \simeq \nn{x}{X}
\end{align} 
for some $\phi\in H^\infty_0(\sect{\rho})\sm\{0\}$. 

Then $A+B$ with $D(A+B)=D(A)\cap D(B)$ is closed, and $A(A+B)^{-1}$ extends to a bounded operator on $X$. 
\end{theorem}

\begin{proof}
  We define the operator $S$ by \eqref{eq:defS-key} and apply
  Lemma~\ref{lem:key}. We may replace $\phi$ in \eqref{eq:DP-sim} by
  $\ph_\pm$, see \cite[Section~6.2]{fc-book}. Then we estimate
\begin{align*}
 &\ \left|\int_0^\infty \langle \psi_\pm(tB) \ph_\pm(tA)x, \big( \ph_\pm(tA)\big)'x' \rangle\,\frac{dt}t \right| \\
 \le &\ \nn{t\mapsto \psi_\pm(tB) \ph_\pm(tA)x}{L^p(0,\infty;\frac{dt}t,X)}
         \,\,\nn{t\mapsto \ph_\pm(tA)\big)'x'}{L^{p'}(0,\infty,\frac{dt}t,X')} \\
 \le &\ \left(\sup_t\n{\psi_\pm(tB)}\right) \nn{t\mapsto\ph_\pm(tA)x}{L^p(0,\infty;\frac{dt}t,X)}
         \,\,\nn{t\mapsto \ph_\pm(tA)\big)'x'}{L^{p'}(0,\infty,\frac{dt}t,X')}.
\end{align*}
By sectoriality of $B$ we have 
\begin{align}\label{eq:B-sect}
 \sup_t\n{\psi_\pm(tB)}<\infty.
\end{align}
Denoting by $\dot{X}_\al(A)$, $\al\in\R$, the completion of $(D(A^\al),\nn{A^\al\cdot}{X})$, condition \eqref{eq:DP-sim} is, for $1\le p<\infty$, equivalent to 
\begin{align}\label{eq:Xreal-ip}
 X=(\dot{X}_{-1}(A),\dot{X}_{1}(A))_{\einhalb,p},
\end{align}
where we recall that $D(A)\cap R(A)=\dot{X}_{-1}(A)\cap \dot{X}_1(A)$ is dense in $X$. For $p=\infty$ this and \eqref{eq:DP-sim} together imply that $X$ equals the closure of $D(A)\cap R(A)$ in $(\dot{X}_{-1}(A),\dot{X}_1(A))_{\einhalb,\infty}$, i.e., for $p=\infty$ we have
\[
 X=(\dot{X}_{-1}(A),\dot{X}_1(A))_{\einhalb,\infty}^0
\]
in the notation of \cite[1.11.2]{Triebel}.
Hence \cite[1.11.2~Theorem]{Triebel} yields
\[
 X'=(\dot{X}_{-1}(A)',\dot{X}_1(A)')_{\einhalb,p'},\quad 1\le p\le\infty.
\]
Since $\big(\dot{X}_\al(A)\big)'=(X')^\cdot_{-\al}(A')$ in a canonical way, the norm in $X'$ is equivalent to expressions
$\nn{t\mapsto (\phi(tA))'x'}{L^{p'}(0,\infty,\frac{dt}t;X')}$ where
$\phi\in H^\infty_0(\sect{\rho})\sm\{0\}$. 
Hence we have
\begin{align*}
 \nn{t\mapsto \ph_\pm(tA)\big)x}{L^{p}(0,\infty,\frac{dt}t,X)} \lesssim & \ \nn{x}{X}, \\
 \nn{t\mapsto (\ph_\pm(tA)\big)'x'}{L^{p'}(0,\infty,\frac{dt}t,X')} \lesssim & \ \nn{x'}{X'},
\end{align*}
and this allows to conclude.
\end{proof}

\begin{remark}\rm\label{rem:DP}
The original result in \cite{DaP-G} is formulated in real interpolation spaces 
$X=(\mathcal{X},\dot{\mathcal{X}}_1(\mathcal{A}))_{\theta,p}$, $\theta\in(0,1)$, where $\mathcal{A}$ is sectorial in $\mathcal{X}$, and $A$ is the ``part'' of $\mathcal{A}$ in $X$. This condition is equivalent to the relation \eqref{eq:Xreal-ip} in the proof as can be seen, e.g., by reiteration.
\end{remark}

\section{The Kalton-Weis result}\label{sec:Ka-W}

To formulate the Kalton-Weis approach we recall some definitions and
concepts.  Let $X$ be a Banach space and $\cT \subseteq \cL(X)$. Then
$\cT$ is called $R$-bounded (resp. $\gamma$-bounded) with constant $C$
if for each $N \ge 1$ and each choice $T_1, \ldots, T_N \in \cT$ and
and each choice $x_1, \ldots, x_N \in X$
\[
  \EE \Bignorm{ \sum_{n=1}^N r_n  T_n x_n }^2
  \le C^2  \; \EE \Bignorm{ \sum_{n=1}^N r_n  x_n }^2
\]
where $(r_n)$ is an i.i.d. sequence of Rademachers or replaced by an 
i.i.d sequence of Gaussians $(\gamma_n)$, respectively.
The case $N=1$ implies
that $R$-bounded or $\gamma$-bounded collections are uniformly
bounded. It is moreover well known that $R$-boundedness implies
$\gamma$-boundedness, see for example \cite[Theorem
8.1.8]{fabulous4:band2}.  We call an operator $A$ $R$-sectorial, respectively 
$\gamma$-sectorial of angle $\rho$, if for each $\eps>0$ the
collection
\[
  \cT = \{ \lambda R(\lambda, A): \lambda \not\in \sect{\rho+\eps} \}
\]
is  $R$-bounded, respectively $\gamma$-bounded.

\medskip

An operator $ T $ acting from a separable\footnote{we only need
  separable Hilbert spaces in our context, see
  \cite[Section~9.1]{fabulous4:band2} for the general case.} Hilbert
space $ H $ into a Banach space $ X $ is called $\gamma$-radonifying
if, for some (and hence all) orthonormal basis $(h_n)$ and an
i.i.d. sequence of Gaussian random variables $(\gamma_n)$ on a
probability space $\Omega$, the series
\[
  \sum_n \gamma_n T h_n
\]
converges in $ L^2(\Omega; X) $. The space of $\gamma$-radonifying operators, equipped with the norm
\[
  \|T\|_\gamma = \left( \mathbb{E} \left\| \sum_n \gamma_n T h_n \right\|^2 \right)^{\einhalb}
\]
forms a Banach space, denoted by $\gamma(H; X)$. If $X$ is Hilbert,
$\gamma$-radonifying operators coincide with Hilbert-Schmidt
operators. Denote by $\gamma'(H; X')$ its dual space, defined by trace
duality \cite[section 9.1.f]{fabulous4:band2}. We are particularly
interested in determining when certain ``square'' functions define
operators in $\gamma(H; X)$ or $\gamma'(H; X')$ in the following way.

\medskip\noindent Let us assume that $A$ is an injective sectorial
operator of angle $\omega$ having dense range. For an interval $I$
equipped with a measure $\mu$ and a nontrivial domain $\mathcal{O}$ in
the complex plane, we can regard a bounded function
$f: I \times \mathcal{O} \to \mathbb{C}$ that satisfy
\begin{enumerate}
\item For all $ t \in I $, the function $ z \mapsto f(t, z) $ is holomorphic on $ \mathcal{O} $.
\item For all $ z \in \mathcal{O} $, the function $ t \mapsto f(t, z) $ is square integrable with respect to $\mu$.
\end{enumerate}
as a vector-valued bounded holomorphic function\footnote{such
  functions a automatically measurable}
$ F: \mathcal{O} \to L^2(I, d\mu)$.  This is in particular the case
for $\phi \in H_0^\infty(\Sigma_\theta)$, $I=(0, \infty)$, $\mu = dt/t$ and
$f(t, z) := \phi(tz)$.  For each $x\in X, x' \in X'$ we may then
define the operators
\[
  T_x: 
  \begin{cases}
    L^2(I) &\to X \\
    g &\mapsto \langle g, F(z) \rangle_{L^2}  (A)x
  \end{cases}
\]
as well as a ``dual'' counterpart, the operator
\[
  T_{x'}^d: 
  \begin{cases}
    L^2(I) &\to X' \\
    g &\mapsto \big(\langle g, F(z) \rangle (A)\big)' x'
  \end{cases}
\]
By a theorem of Kalton and Weis,  these operators are tightly linked to
the boundedness of the $H^\infty$-calculus: if both operators, 
\[
  S: \quad x \mapsto T_x \in \cL(X;  \gamma(H; X))
  \quad\text{and}\quad
  S^d: \quad x' \mapsto T_{x'}^d \in \cL(X';  \gamma'(H; X'))
\]
are bounded, then $A$ admits a bounded functional calculus on
$\Sigma_\theta$, and conversely, if $X$ is a Banach space of finite
cotype, then boundedness of the $H^\infty$-functional calculus on
$\Sigma_\nu$ for some $\nu < \theta$ implies boundedness of the
operators $S, S^d$.  We refer e.g. to \cite[Theorems 7.2 and
7.7]{KalWei2004}, or to \cite[Theorems 10.4.16 and
10.4.19]{fabulous4:band2} for details and references.  If $S$
(respectively, $S^d$) is bounded, we say that $A$ admits square
function estimates (respectively, dual square function
estimates). Upon proper use of vector-valued (Pettis)-integrals, a
Fubini argument allows to write the operators $T_x$ and $T^d_{x'}$ as
\[
   x \mapsto \phi(tA) x
   \quad\text{and}\quad
   x' \mapsto  \phi(tA)'x'
 \]
 respectively, see \cite[Chap. 9]{fabulous4:band2}. We do not go
 into details of these vector-valued integrals, we consider them here
 as a handy symbolic notation for the previously constructed operators
 $T_x$ and $T^d_{x'}$.

\begin{theorem}[Kalton-Weis]
  Let Assumption~\assref{ass:main} hold. In addition, let $A$ and $B$ be
  $\ga$-sectorial in $X$ of angles $\om_A^\ga$ and $\om_B^\ga$,
  respectively, and assume $\om_A^\ga+\om_B^\ga<\pi$ and
  $\rho\in(\om_A^\ga,\pi-\om_B^\ga)$. Assume that $A$ has 
  square function estimates and dual square function estimates associated
  with the functions $\varphi_\pm$.

   Then $A+B$ with $D(A+B)=D(A)\cap D(B)$ is closed, and $A(A+B)^{-1}$ extends to a bounded operator on $X$. 
\end{theorem}

\noindent By the preceding discussion, the assumptions on $A$ are satisfied
if $A$ has a bounded $H^\infty$-functional calculus on a Banach space $X$ having finite cotype.

\begin{proof}
  We define the operator $S$ by \eqref{eq:defS-key} and apply
  Lemma~\ref{lem:key}.  Then we estimate
\begin{align*}
 &\ \left|\int_0^\infty \langle \psi_\pm(tB) \ph_\pm(tA)x, \big( \ph_\pm(tA)\big)'x' \rangle\,\frac{dt}t \right| \\
 \le &\ \nn{t\mapsto \psi_\pm(tB) \ph_\pm(tA)x}{\ga(0,\infty;\frac{dt}t,X)}
         \,\,\nn{t\mapsto \ph_\pm(tA)\big)'x'}{\ga(0,\infty,\frac{dt}t,X')} \\
 \le &\ \ga\left(\{\psi_\pm(tB):t>0\}\right) \nn{t\mapsto\ph_\pm(tA)x}{\ga(0,\infty;\frac{dt}t,X)}
         \,\,\nn{t\mapsto \ph_\pm(tA)\big)'x'}{\ga(0,\infty,\frac{dt}t,X')}.
\end{align*}
By $\ga$-sectoriality of $B$ and a standard argument on closed convex
combinations of $\gamma$-bounded sets, we have
\begin{align}\label{eq:B-ga-sect}
  \ga\left(\{\psi_\pm(tB):t>0\}\right) <\infty,
\end{align}
Now
\begin{align*}
 \nn{t\mapsto \ph_\pm(tA)\big)x}{\ga(0,\infty,\frac{dt}t,X)} \lesssim & \ \nn{x}{X}, \\
 \nn{t\mapsto \ph_\pm(tA)\big)'x'}{\ga(0,\infty,\frac{dt}t,X')} \lesssim & \ \nn{x'}{X'}
\end{align*}
allow to conclude.
\end{proof}

\section{Closed operator sums in generalized Triebel-Lizorkin spaces}\label{sec:Tr-L}

We give a result on closed operator sums in the generalized
Triebel-Lizorkin spaces introduced in
\cite{KUll}. Theorem~\ref{thm:closed-lq} may be seen as a counterpart
to the Da~Prato-Grisvard result with real interpolation replaced by
the $\ell^q$-interpolation method from \cite{Ku:lq}. This needs some
preparation. Recall the notation  $L^0(\Om)$ for the equivalence classes
of $(\Om,\mu)$ measurable functions.

\begin{assumption}[BFS]\label{ass:BFS}
  We say that a Banach space $X$ satisfies the
  assumption~\assref{ass:BFS} if $X$ is a Banach function space in the
  sense of \cite{LN-BFS} that has the Fatou property as well as an
  absolutely continuous norm.

  \noindent Recall that a Banach space $(X,\nn{\cdot}{X})$ is a \emph{Banach
    function space} if there exists a
  $\si$-finite measure space $(\Om,\mu)$ with $X\subseteq L^0(\Om)$
  such that
  \begin{enumerate}[label=(\roman*)]
  \item\label{BFS-item-i}
    If $f\in X$ and $g\in L^0(\Omega)$ with $|g|\le|f|$ a.e.,
    then $g\in X$ and $\|g\|_{X}\le\|f\|_{X}$.
  \item\label{BFS-item-ii}
    For every measurable $E\subseteq\Omega$ of positive measure,
    there exists a measurable subset $F\subseteq E$ of positive
    measure with $\eins_F\in X$.
\end{enumerate}
Recall also the notions of the Fatou property
\begin{enumerate}[label=(\roman*),resume]
  \item\label{BFS-item-iii}
    If $(f_n)$ is a sequence in $X$ with $0\le f_n\le f_{n+1}$ and
    $\sup_n\|f_n\|_{X}<\infty$, then $f=\sup_nf_n\in X$ and
    $\|f\|_{X}=\sup_n\|f_n\|_{X}$.
  \end{enumerate}
  finally, recall the notion of absolutely continuous norms for Banach
  function spaces:
\begin{enumerate}[label=(\roman*),resume]
\item\label{BFS-item-iv}
    If $f\in X$ and $(E_n)$ is a decreasing sequence of measurable
    subsets with $\eins_{E_n}\to0$ a.e., then
    \(\|f\,\eins_{E_n}\|_{X}\to0\).
\end{enumerate}
\end{assumption}

Assume that $X$ is such a Banach function space. By, e.g.,
\cite[Prop.~3.12]{LN-BFS}, for a space $X$ with \ref{BFS-item-i} and \ref{BFS-item-ii},
\ref{BFS-item-iv} is equivalent to order-continuity of $X$, and by
\cite[Prop.~3.15]{LN-BFS} this is equivalent to
\[
 X'=\{g\in L^0(\Om): \forall f\in X: fg\in L^1(\Om)\,\}.
\]
In this case, we have by \cite[Prop.~3.1]{LN-BFS} that
\[
 \nn{g}{X'}=\sup\{\int_\Om|fg|\,d\mu: f\in X, \nn{f}{X}\le1\}.
\]
We also remark that, by \cite[Cor.~3.16]{LN-BFS}, a space $X$
satisfying Assumption~\assref{ass:BFS} is reflexive if and only if $X'$
is order-continuous.
Via \cite[Prop.~2.5]{LN-BFS} we can see that Assumption~\assref{ass:BFS}
on $X$ is equivalent to those used in  \cite{Ku:lq,KUll}.

Let $A$ be a sectorial operator of angle $\om_A\in[0,\pi)$ in $X$ with
dense domain and range and let $q\in[1,\infty)$. Then $A$ is
\emph{$\cR_q$-sectorial of angle $\om^{\ell^q}_A\in[\om_A,\pi)$} if, for
any $\om\in(\om_A^{\ell^q},\pi)$, the set
$\{\la R(\la,A):\la\in\C\sm\sect{\om}\}$ is \emph{$\ell^q$-bounded} in
$\cL(X)$. Here, a subset $\cT\subseteq\cL(X)$ is called
\emph{$\ell^q$}-bounded if there exists a constant $C>0$ such that, for
all $n\in\N$ and $x_j\in$, $T_j\in\cT$, $j=1,\ldots,n$, we have
\[
 \Bignorm{\Big(\sum_{j=1}^n|T_jx_j|^q\Big)^{1/q}}_{X} \le C \Bignorm{\Big(\sum_{j=1}^n|x_j|^q\Big)^{1/q}}_{X}.
\]
For simplicity, we state our theorem explicitly for the generalized
Triebel-Lizorkin spaces $X_{\theta,\ell^q}(A)$, $\theta\in(0,1)$,
associated with an $\cR_q$-sectorial operator $A$ in $X$ (introduced
in \cite{KUll} with a different notation) in the case
$0\in\rho(A)$. This is more in the spirit of the original formulation
of the Da~Prato-Grisvard result. More general situations will be
studied elsewhere. In case $0\in\rho(A)$, the spaces from \cite{KUll}
are given by
\[
 X_{\theta,\ell^q}(A):=\{x\in X: t\mapsto t^{-\theta}\phi(tA)x\in X(L^q_*)\}, \quad \theta>0,
\]
where $L^q_*:=L^q(0,\infty;\frac{dt}t)$, with norm
\[
 \nn{x}{X_{\theta,\ell^q}(A)}:=\bignorm{t\mapsto t^{-\theta}\phi(tA)x}_{X(L^q_*)}
 =\Bignorm{\Big(\int_0^\infty |t^{-\theta}\phi(tA)x|^q\,\frac{dt}t\Big)^{1/q}}_{X}.
\]
Here, $\phi\neq0$ is a function in
$H^\infty_0(\sect{\om})$ such that
$z\mapsto z^{-\theta}\phi(z)\in H^\infty_0(\sect{\om})$.

\begin{theorem}\label{thm:closed-lq}
Suppose that $X$ satisfies Assumption~\assref{ass:BFS} and is reflexive.
Let Assumption~\assref{ass:main} hold. In addition, let $A$ and $B$ be $\cR^q$-sectorial in $X$ of angles $\om_A^{\ell^q}$ and $\om_B^{\ell^q}$, respectively, and assume $\om_A^{\ell^q}+\om_B^{\ell^q}<\pi$ and $\rho\in(\om_A^{\ell^q},\pi-\om_B^{\ell^q})$.
In addition, suppose $0\in\rho(A)$. 
Denote by $A_\theta$ and $B_\theta$ the parts of $A$ and $B$, respectively, in $X_{\theta,\ell^q}(A)$. 

Then $A_\theta+B_\theta$ with $D(A_\theta+B_\theta)=D(A_\theta)\cap D(B_\theta)$ is closed in $X_{\theta,\ell^q}(A)$, and $A_\theta(A_\theta+B_\theta)^{-1}$ extends to a bounded operator on $X_{\theta,\ell^q}(A)$. 
\end{theorem}

\noindent We remark that, in the situation of Theorem~\ref{thm:closed-lq}, we
have
\[
 D(A_\theta)=A^{-1}\big(X_{\theta,\ell^q}(A)\big)=X_{\theta+1,\ell^q}(A)\quad\mbox{and}\quad
 D(B_\theta)=(1+B)^{-1}\big(X_{\theta,\ell^q}(A)\big).
\]

\begin{proof}
  We first note that Assumption~\assref{ass:main} also holds for
  $X_{\theta,\ell^q}(A)$ and $A_\theta$, $B_\theta$ in place of $X$, and
  $A$, $B$.  We then define the operator $S$ by \eqref{eq:defS-key}
  with $A_\theta$, $B_\theta$ in place of $A$, $B$, and shall apply
  Lemma~\ref{lem:key}. We represent duality between
  $X_{\theta,\ell^q}(A)$ and $(X_{\theta,\ell^q}(A))'$ via
  Proposition~\ref{prop:dual-TL} in Appendix~\ref{app:dual-lq}. For
  nice enough $x\in D(A)\cap D(S)$ and $x'\in D(A')$ we can thus work
  in $\dual{\cdot}{\cdot}=\idual{\cdot}{\cdot}{X,X'}$.
  Then we estimate
\begin{align*}
 &\  \left|\int_0^\infty \langle \psi_\pm(tB) \ph_\pm(tA)x, \big( \ph_\pm(tA)\big)'x' \rangle\,\frac{dt}t \right|    \\
 =&\ \left|\int_0^\infty \langle \psi_\pm(tB) \ph_\pm(tA)A^\theta x, \big( \ph_\pm(tA)\big)'(A')^{-\theta} x' \rangle\,\frac{dt}t \right| \\
 \le &\ \bignorm{t\mapsto \psi_\pm(tB) \ph_\pm(tA)A^\theta x}_{X(L^q_*)}
         \,\,\bignorm{t\mapsto (\ph_\pm(tA)\big)'(A')^{-\theta}x'}_{X'(L^{q'}_*)} \\
 \le &\ \mathcal{R}_q\left(\{\psi_\pm(tB):t>0\}\right) \bignorm{t\mapsto\ph_\pm(tA)A^\theta x}_{X(L^q_*)}
         \,\,\bignorm{t\mapsto (\ph_\pm(tA)\big)'(A')^{-\theta}x'}_{X'(L^{q'}_*)}.
\end{align*}
Since $A$ is $\cR_q$-sectorial in $X$ we may, by
\cite[Prop.~4.5]{KUll} and at the expense of a constant, change
$\ph_\pm$ in the last line to $\phi\in H^\infty_0$ given by
$\phi(z):=z^2(1+z)^{-4}$, i.e. we obtain
\[
 \lesssim  \mathcal{R}_q\left(\{\psi_\pm(tB):t>0\}\right) \bignorm{t\mapsto\phi(tA)A^\theta x}_{X(L^q_*)}
         \,\,\bignorm{t\mapsto (\phi(tA)\big)'(A')^{-\theta}x'}_{X'(L^{q'}_*)}.
\]
By $\cR^q$-sectoriality of $B$ we have 
\begin{align}\label{eq:B-lq-sect}
 \mathcal{R}_q\left(\{ \psi_\pm(tB) :t>0\}\right) <\infty.
\end{align}
We define $\phi_{\theta},\phi_{1-\theta}\in H^\infty_0$ by $\phi_{\theta}(z):=z^{\theta}\phi(z)$, $\phi_{1-\theta}(z)=z^{1-\theta}\phi(z)$ and we rewrite
\begin{align*}
 \phi(tA)A^\theta x =&\  t^{-\theta} (tA)^\theta \phi(tA) x = t^{-\theta} \phi_\theta(tA) x \\
 \phi(tA')(A')^{-\theta} x' =&\  t^{\theta-1} (tA')^{1-\theta} \phi(tA') (A')^{-1} x' 
 = t^{\theta-1} \phi_{1-\theta}(tA') (A')_{-1}^{-1} x'. 
\end{align*}
Then we observe that
\begin{align*}
 \bignorm{t\mapsto t^{-\theta} \phi_\theta(tA)x}_{X(L^q_*)} \simeq \nn{x}{X_{\theta,\ell^q}(A)}
\end{align*}
by definition and 
 \begin{align*}
 \bignorm{t\mapsto t^{\theta-1} \phi_{1-\theta}(tA') (A')_{-1}^{-1} x'}_ {X'(L^{q'}_*)} \simeq \nn{x'}{(X_{\theta,\ell^q}(A))'}
\end{align*}
by Proposition~\ref{prop:dual-TL}. This finishes the proof.
\end{proof}

\begin{remark}\label{eq:TL-rem}
  In the case $q=2$, by Khintchine and Kahane's inequality,
  $X(L^2(I)) \simeq \gamma(I, X)$ so that 
  Theorem~\ref{thm:closed-lq} recovers the Kalton-Weis result.
\end{remark}

We illustrate our result with an application in a maximal regularity context.

\begin{prop}
  Let $E$ be a reflexive Banach function space (see
  Assumption~\assref{ass:BFS}), let $1<q<\infty$, and let $A_0$ be an
  $\cR_q$-sectorial operator in $E$ with dense domain and range,
  $0\in\rho(A_0)$, and $\om_{A_0}^{\ell^q}<\pi/2$. Let $X=E(L^q(\R_+))$
  and, for fixed $\theta\in(0,1)$, let
\[
Y_\theta:=\{f:\Om\times\R_+\to\C:
 \nn{f}{Y_\theta}:=\Bignorm{\left(\int_0^\infty\int_0^\infty |s^{-\theta}\ph(sA)(f(\cdot,t))|^q\,dt\frac{ds}s\right)^{1/q} }_{E}<\infty\},
 \]
 where the $f$ are measurable. Denote by $(A_0)_\theta$ the part of $A_0$ in
 $E_{\theta,\ell^q}(A_0)$.
 Then, for any $f\in Y_\theta$, the abstract Cauchy problem
 \[
 u'(t)+(A_0)_\theta u(t)= f(t), \quad t>0,\qquad u(0)=0,
 \]
has a unique mild solution $u$, and this solution satisfies $u',(A_0)u\in Y_\theta$. 
\end{prop}

The interpretation is that $(A_0)_\theta$ has maximal
$Y_\theta$-regularity where we understand $Y_\theta$ as a realization
of $E_{\theta,\ell^q}(A_0)(L^q(\R_+))$.

\begin{proof}
  Step 1: By \cite[Prop. III.3.16]{DS:band1} elements of $X$ have
  measurable representants on $\Om\times\R_+$, we check that $X$ is
  reflexive and is a Banach function space, according to
  Assumption~\assref{ass:BFS}.

  \noindent Property \ref{BFS-item-i} is clear. By \cite[Prop.~2.5]{LN-BFS} the saturation
  property \ref{BFS-item-ii} for $E$ is equivalent to the existence of an
  increasing sequence $(F_n)$ of measurable subsets of $\Om$ with
  $\eins_{F_n}\in E$ and $\bigcup_n F_n=\Om$. Then $(F_n\times(1/n,n))$ is
  increasing and satisfies $\eins_{F_n\times(1/n,n)}\in X$ and
  $\bigcup_n (F_n\times(1/n,n))=\Om\times(0,\infty)$, which proves
  \ref{BFS-item-ii} for $X$. Since $E$ is reflexive and $q\in(1,\infty)$ we have
  $X'=E'(L^{q'}(\R_+))$ and see easily that $X$ is reflexive. By
  \cite[Cor.~3.16]{LN-BFS}, both $X$ and $X'$ have the Fatou property
  \ref{BFS-item-iii} and are order-continuous. By \cite[Prop.~3.12]{LN-BFS}
  again, $X$ (and $X'$) have absolute continuous norm, i.e. they
  satisfy \ref{BFS-item-iv}.

\smallskip\noindent Step 2: Define the linear operator $A$ in $X$ by
$(Au)(\om,t):=(A_0 u(\cdot,t))(\om)$ for $u\in D(A)$ where
\[
 D(A)=\{u\in X: u(\cdot,t)\in D(A_0)\ \mbox{for a.e. $t\in\R_+$}, Au\in X\}.
\]
We check that $A$ is $\cR_q$-sectorial in $X$ with
$\om_{A}^{\ell^q}=\om_{A_0}^{\ell^q}<\pi/2$: By $\cR_q$-sectoriality of
$A_0$ in $E$ (used for singletons) we first see that $A$ is sectorial
in $X$ of angle $\om_{A_0}^{\ell^q}$ and that this angle is optimal. Now
let $\si\in(\om_{A_0}^{\ell^q},\pi)$ and take a finite number of
$\la_j\in\C\sm\Si_\si$ and functions $f_j$ of the form
\[
 f_j(\cdot, t)=\sum_k f_{jk} \eins_{I_{jk}}(t),
 \]
 where each sum is finite, the $(I_{jk})_k$ are bounded pairwise disjoint intervals in $(0,\infty)$, and $f_{jk}\in E$. Then we have
\begin{align*}
  \bignorm{\big(\sum_j|\la_j R(\la_j,A)f_j|^q\big)^{1/q}}_{X}
  & =\ \Bignorm{\Big(\int_0^\infty\sum_j
    |\la_j R(\la_j,A_0)f_{jk} \eins_{I_{jk}}(t)|^q\,dt\Big)^{1/q}}_{E} \\
  & =\ \Bignorm{\Big(\sum_{j,k}
    \big|\la_j R(\la_j,A_0)\big(|I_{jk}|^{1/q}f_{jk}\big)\big|^q \Big)^{1/q}}_{E} \\
  & \le\ \cR_q\big(\{\la R(\la,A_0):\la\in\C\sm\Si_\si\,\}\big)
    \Bignorm{\Big(\sum_{j,k}
    \big|\big(|I_{jk}|^{1/q}f_{jk}\big)\big|^q \Big)^{1/q}}_{E} \\
  &=\ \cR_q\big(\{\la R(\la,A_0):\la\in\C\sm\Si_\si\,\}\big)  \bignorm{\big(\sum_j|f_j|^q\big)^{1/q}}_{X}.
\end{align*}
From this the claim follows by approximation and the Fatou property.

\smallskip\noindent Step 3: Define $B=\frac{d}{dt}$ in $X$ by $(Bu)(\om,t):=(\partial_t u(\om,\cdot))(t)$ for $u\in D(B)$ where
\[
 D(B)=\{u\in X: u(\om,\cdot)\in W^{1,q}(\R_+), u(\om,0)=0\
 \mbox{for a.e. $\om\in\Om$}, Bu\in X\}.
\]
We claim that $B$ is $\cR_q$-sectorial in $X$ with $\om_B^{\ell^q}=\pi/2$. So let $\si\in(0,\pi/2)$. First we observe that, for $\la\in\Si_\si\sm\{0\}$ and $g\in L^q(\R_+)$, the solution of $\la f(t)+f'(t)=g(t)$, $t\ge0$, with $f(0)=0$ is given by $g(t)=\int_0^t e^{-\la(t-s)}g(s)\,ds$, $t\ge0$, i.e. by convolution of $g$ with $e^{-\la(\cdot)}\in L^1(\R_+)$. For a finite family $(\la_j)$ in $\Si_\si\sm\{0\}$ and $f_j\in X$ we thus have
\begin{align*}
  \Bignorm{ \Big(\sum_j|\la_j(\la_j+B)^{-1}f_j|^q\Big)^{1/q}}_{X}
  &=\ \Bignorm{\Big(\int_0^\infty
    \sum_j|\la_j (e^{-\la_j(\cdot)}*f_j)(t) |^q\,dt\Big)^{1/q}}_{E} \\
  &\le\ \Bignorm{\Big(\int_0^\infty
    \sum_j\big(|\la_j| e^{-\Rep\la_j(\cdot)}\big)*|f_j|(t) |^q\,dt\Big)^{1/q}}_{E} \\
  & \le\ \frac1{\cos\si}\,
  \Bignorm{\Big(\sum_j\int_0^\infty (Mf_j)(t)^q\,dt\Big)^{1/q}}_{E},
\end{align*}
where $M$ denotes the Hardy-Littlewood maximal operator on $\R_+$. Boundedness of $M$ on $L^q(\R_+)$ (recall $1<q<\infty$) allows to conclude.

\smallskip\noindent Step 4: It is nearly obvious that resolvents of
$A$ and $B$ commute on $X$.  An application of
Theorem~\ref{thm:closed-lq} now yields that
$A_\theta(A_\theta+B_\theta)^{-1}$ and
$B_\theta(A_\theta+B_\theta)^{-1}$ are bounded in
$X_{\theta,\ell^q}(A)$. This yields the claim since
$X_{\theta,\ell^q}(A)=Y_\theta$ by Fubini.
\end{proof}

\section{The Dore-Venni result}\label{sec:Do-V}

Let $A$ be a sectorial operator of angle $\om_A\in[0,\pi)$ in a Banach
space $X$ with dense domain and range. Then $A$ is said to have
\emph{bounded imaginary powers} (BIP) if there exist constants $M>0$
and $\theta\ge0$ such that
\[
 \nn{A^{it}}{\cL(X)}\le M e^{\theta|t|},\quad t\in\R.
\]
We denote by $\theta_A\ge0$ the infimum of all such $\theta$ and call it the \emph{BIP angle} of $A$. We always have $\theta_A\ge\om_A$, see e.g. \cite[Cor.~4.3.4]{fc-book}

We formulate the result we shall prove. The original result by Dore
and Venni (\cite{Do-V}) assumed $0\in\rho(A)\cap\rho(B)$.  This
assumption has been removed by Pr\"uss and Sohr (\cite{Pr-So}). The
proof in \cite{Pr-So} is done by considering $\eps+A$ and $\eps+B$ and
their imaginary powers and then letting $\eps\to0+$. The key of the
proof is then, of course, to obtain estimates that are uniform in
$\eps\in(0,1)$.  There is another proof by Haase (\cite{Ha:groups}),
based on functional calculus for commuting groups and strip type
operators and building on previous work by Monniaux (\cite{Mo}), who
used so-called analytic generators.

The proof we shall give below is quite different. In line with the
general idea of our approach we are emphasizing the aspects common
with the other operator sum results. Here, the related interpolation
method is the complex method. In order to see this, we have to rely on
the results on BIP, the complex method, the Mellin transform
$\mathscr{M}$, and the Gamma function, presented in the appendix. In
fact, $A$ having BIP is equivalent to the analog of
\eqref{eq:normXeqsim}, see Proposition~\ref{prop:BIP-sqf}. The BIP
property for $B$ and the UMD property of $X$ together are needed to
verify the analog of \eqref{eq:B-bdd} in the proof, which is done in
Lemma~\ref{lem:G*-L2bdd} below. Thus these conditions together
correspond to conditions \eqref{eq:B-sect}, \eqref{eq:B-ga-sect}, or
\eqref{eq:B-lq-sect} in the previous sections.

Unfortunately but also unavoidably, somehow ``again the Mellin
transform [is] all over the place'' (\cite[p.~530]{Ha:groups}), but we
do not need to know much, which justifies to call it ``cheap'' again.

\begin{theorem}\label{thm:DV}
  Let Assumption~\assref{ass:main} hold. In addition, let 
  $A$ and $B$ have BIP of angles $\theta_A$ and $\theta_B$, respectively, 
  with $\theta_A+\theta_B<\pi$ and assume $\rho\in(\theta_A,\pi-\theta_B)$.
  Assume that $X$ is a UMD space. 

  Then $A+B$ with domain $D(A+B)=D(A)\cap D(B)$ is a closed, and
  $A(A+B)^{-1}$ extends to a bounded operator on $X$. 
\end{theorem}

\begin{proof}
  We define the operator $S$ by \eqref{eq:defS-key} and apply
  Lemma~\ref{lem:key}.  We have to show the estimate
  \eqref{eq:key-lem-est}. We shall only show the estimate for
\begin{align}\label{eq:DV+case}
 \left|\int_0^\infty \langle \psi_+(tB) \ph_+(tA)x, 
                                                \big( \ph_+(tA)\big)'x' \rangle\,\frac{dt}t \right|,       
\end{align}
in detail, the other case with $(\psi_-,\ph_-)$ in place of
$(\psi_+,\ph_+)$ being similar.  By \eqref{eq:dual-mellin}, the term
\eqref{eq:DV+case} equals, up to a constant,
\begin{align}\label{eq:L2-est-prod}
    & \left|\int_\R \langle \mathscr{M}(\psi_+(tB)\ph_+(tA)x)(i\si), \mathscr{M}(\ph_+(tA)'x')(-i\si)\rangle \,d\si\right| \notag \\
\le & \nn{ \mathscr{M}(\psi_+(tB)\ph_+(tA)x)(i\cdot)}{L^2(\R;X)} \nn{ \mathscr{M}(\ph_+(tA)'x')(-i\cdot)}{L^2(\R;X')}.
\end{align}
By Lemma~\ref{lem:mellin-ph} we have
\begin{align*}
 \mathscr{M}(\ph_+(tA)x)(i\si)& 
 =\frac{(-e^{i\rho})^{i\si-1/4}}{\sqrt{\pi}} \Ga(\mbox{$\frac14$}+i\si)\Ga(\mbox{$\frac14$}-i\si) A^{-i\si}x \\
 & = e^{-i(\pi-\rho)/4} \, e^{-\si(\pi-\rho)} \, \Ga(\mbox{$\frac14$}+i\si)\Ga(\mbox{$\frac14$}-i\si) A^{-i\si}x.
\end{align*}
Since $A$ has BIP in $X$ of angle $\theta_A$, we obtain via Lemma~\ref{lem:Gamma-funktion} that
\begin{align*}
 \nn{\mathscr{M}(\ph_+(tA)x)(i\si)}{X}\lesssim e^{-\rho|\si|} \, e^{\tilde{\theta}_A|\si|} \, \nn{x}{X},
\end{align*}
where $\tilde{\theta}_A-\theta_A>0$ is as small as we wish. By choice of $\rho>\theta_A$ we thus conclude
\begin{align}\label{eq:L2-est-phi}
 \nn{ \mathscr{M}(\ph_+(tA)x)(i\cdot)}{L^2(\R;X)} \lesssim \nn{x}{X}.
\end{align}
Since $A'$ has BIP in $X'$ of angle $\theta_A$ as well, we conclude similarly 
\begin{align}\label{eq:L2-est-dual}
 \nn{ \mathscr{M}(\ph_+(tA)'x')(-i\cdot)}{L^2(\R;X')} \lesssim \nn{x'}{X'}.
\end{align}
In order to estimate the first factor in \eqref{eq:L2-est-prod} we
recall \eqref{eq:mellin-conv} : we have to convolve the $X$-valued
function $\mathscr{M}(\ph_+(tA)x)(i\cdot)$ with the $\cL(X)$-valued
function $\mathscr{M}(\psi_+(tB))(i\cdot)$.  By
Lemma~\ref{lem:mellin-ph} we have
\[
 \mathscr{M}(\psi_+(tB)x)(i\si) = 
  e^{i\rho(i\si-1)}\,\frac{\pi}{\sin(\pi i\si)} B^{-i\si}x  
  = e^{-i\rho} \, e^{-\rho\si} \frac\pi{i \sinh(\pi\si)} B^{-i\si}x.
\]
Hence the strongly continuous operator-valued function
\begin{equation}\label{eq:def-G}
  G:\R\sm\{0\} \to \cL(X),\quad \si\mapsto\mathscr{M}(\psi_+(tB))(i\si)
    = e^{-i\rho} \, e^{-\rho\si} \frac\pi{i \sinh(\pi\si)} B^{-i\si}
\end{equation}
has a singularity at the origin. Nevertheless, the following
Lemma~\ref{lem:G*-L2bdd} allows to finish the proof. Of course, this
is where the UMD assumption on $X$ enters.
\end{proof}

\begin{lemma}\label{lem:G*-L2bdd}
  Under the assumptions of Theorem~\ref{thm:DV} and for
  $\rho\in(\theta_A,\pi-\theta_B)$ let $G$ be given by
  \eqref{eq:def-G}. Then the operator $f\mapsto G*f$ is bounded on
  $L^2(\R;X)$.
\end{lemma}
\begin{proof}
  We write
  $G=G\,\eins_{[-1,1]}+G\,\eins_{\R\sm[-1,1]}=:G_0+G_1$. Since, for
  $|\si|\ge1$, we have
\begin{align*}
 \nn{G(\si)}{\cL(X)}\lesssim e^{\rho|\si|} e^{-\pi|\si|} e^{\tilde{\theta}_B|\si|}, 
\end{align*}
where $\tilde{\theta}_B-\theta_B>0$ is as small as we wish, the
function $G_1$ is by choice of $\rho$ strongly in $L^1$, and thus
convolution with $G_1$ is bounded on $L^2(\R;X)$.

\smallskip\noindent We turn to $G_0$. Since $X$ is UMD, convolution
with $\si\mapsto \frac1{\pi\si}\eins_{[-1,1]}(\si)$ is bounded on
$L^2(\R;X)$. Since $\frac1{\pi\si}-\frac{e^{-\rho\si}}{\sinh(\pi\si)}$
is smooth at the origin, also convolution with
\begin{align*}
 g_0:\R\sm\{0\}\to\R, \si\mapsto  g_0(\si) = \pi e^{-i\rho}\frac{e^{-\rho\si}}{\sinh(\pi\si)}\,\eins_{[-1,1]}(\si),
\end{align*}
is bounded on $L^2(\R;X)$. Now we write, for $\tau\in\R$ and $l\in\Z$
with $-\frac12\le\tau-l<\frac12$,
\begin{align*}
  G_0*f(\tau)=
  &\int G_0(\tau-\si)f(\si)\,d\si
    = \int g_0(\tau-\si)B^{-i(\tau-\si)}f(\si)\,d\si  \\
 =& B^{-i(\tau-\ell)}\int g_0(\tau-\si) B^{i(\si-\ell)}f(\si)\,d\si \\
 =& \sum_{j=-1}^1 B^{-i(\tau-\ell)}\int g_0(\tau-\si) B^{i(\si-\ell)} f_{l+j}(\si)\,d\si,
\end{align*}
where $f_k:=f \eins_{Q_k}$ and $Q_k=[k-\einhalb,k+\einhalb)$ for each
$k\in\Z$. We thus have, for $\ell\in\Z$,
\begin{align*}
\nn{G_0*f}{L^2(Q_\ell;X)}^2\lesssim&\sum_{j=-1}^1\nn{B^{-i((\cdot)-\ell)}(g_0*(B^{i((\cdot)-\ell)}f_{\ell+j}))}{L^2(Q_\ell;X)}^2 \\
\lesssim&\sum_{j=-1}^1\nn{f_{\ell+j}}{L^2(Q_{\ell+j};X)}^2.
\end{align*}
Summing over $\ell\in\Z$ we obtain
\begin{align*}
 \nn{G_0*f}{L^2(\R;X)}\lesssim  \nn{f}{L^2(\R;X)},
\end{align*}
which finishes the proof.
\end{proof}

\begin{remark}\rm
Note that the argument to exchange non-zero functions $\phi\in H^\infty_0$ in Littlewood-Paley norms is different here from those in the previous sections.
There the arguments allow to include an operator $F(A)$ for $F\in H^\infty$ and to show at the same time that $A$ has a bounded $H^\infty$-calculus in the respective interpolation spaces. This can not be possible for the complex interpolation method, as this would imply, via similarity of spaces $\dot{D}_\al(A)$ and $X$, that an operator $A$ with BIP of angle $\theta_A<\pi$ always has a bounded $H^\infty$-calculus in $X$, which is known to be false. 
\end{remark}

\appendix

\section{Duals of generalized Triebel-Lizorkin spaces}\label{app:dual-lq}

We shall combine methods from \cite{Ku:lq} and \cite{KUll}. By
\cite{Ku:lq}, the generalized Triebel-Lizorkin spaces from \cite{KUll}
can be obtained by $\ell^q$-interpolation. We are interested in the dual
spaces of the former. More generally, we shall first describe the dual
space for a general interpolation couple $(\cX_0,\cX_1)$ where, for
$\nu=0,1$, the structured Banach spaces $\cX_\nu$ have the form
$\cX_\nu=(X_\nu,J_\nu,E_\nu)$, where $E_\nu$ satisfies
~\assref{ass:BFS} and $J_\nu:X_\nu\to E_\nu$ is an isometric
\emph{surjection}.  This special case covers what we actually need in
Section~\ref{sec:Tr-L} (in \cite{Ku:lq}, the $J_\nu$ are required to
be isometric, but not necessarily surjective). We start with some
preparations.

\begin{lemma}
  If $E$ satisfies Assumption~\assref{ass:BFS} then its dual space $E'$
  equals, in a canonical way, its K\"othe dual
  \[
    \{ g\in L^0(\Om): \; \forall f \in E\quad fg \in L^1(\Om) \}
  \]
  If $E$ is in addition reflexive then also its dual space $E'$
  satisfies Assumption~\assref{ass:BFS}.
\end{lemma}

For the proof we refer to \cite[Prop.~3.1~and~Cor.~3.16]{LN-BFS}.

\begin{lemma}
  Let $\cX=(X,J,E)$ be a structured Banach space where $E$ satisfies
  Assumption~\assref{ass:BFS}, $J:X\to E$ is an isometric surjection
  and $X$ is reflexive. Then $\cX':=(X',(J^{-1})',E')$ is again a
  structured Banach space of this form.
\end{lemma}

\begin{theorem}\label{thm:dual-lq}
  Let $(\cX_0,\cX_1)$ be an interpolation couple of structured Banach
  spaces $\cX_\nu=(X_\nu,J_\nu,E_\nu)$, $\nu=0,1$, where the $E_\nu$
  satisfy Assumption~\assref{ass:BFS} and the $X_\nu$ are
  reflexive. Let $\theta\in(0,1)$.  For the duality pairing
  $\langle\cdot,\cdot\rangle_{X_0\cap X_1,X_0'+X_1'}$ we have
\[
 \big((\cX_0,\cX_1)_{\theta,\ell^q}\big)'=(\cX_0',\cX_1')_{\theta,\ell^{q'}}.
\]
\end{theorem}   

\begin{proof} We follow the lines of the proof of \cite[1.11.2
  Theorem]{Triebel}, the corresponding result for the real
  interpolation method. First let
  $\phi\in(\cX_0',\cX_1')_{\theta,\ell^{q'}}\subseteq X_0'+X_1'$ and
  $x\in X_0\cap X_1\subseteq (\cX_0,\cX_1)_{\theta,\ell^q}$. By
  \cite[Def. 2.9]{Ku:lq} we decompose $x=x_0(j)+x_1(j)$, $j\in\Z$,
  with $x_\nu(j)\in X_\nu$ and
\begin{equation}\label{eq:lq-Kdec}
 \sum_{\nu=0}^1\bignorm{ (2^{(\nu-\theta)j} J_\nu x_\nu(j))_j }_{E_\nu(\ell^q(\Z))} 
 \le 2 \nn{x}{(\cX_0,\cX_1)_{\theta,\ell^q}}<\infty,
\end{equation}
and 
we decompose $\phi=\sum_{j\in\Z}\phi_j$ (convergence in $X_0'+X_1'$)
where $\phi_j\in X_0'\cap X_1'$.
We then have 
\begin{align*}
 |\idual{x}{\phi}{X_0\cap X_1,X_0'+X_1'}|
 \le&\ \sum_{\nu=0}^1\Bigabs{\sum_j \idual{x_\nu(j)}{\phi_j}{X_0\cap X_1,X_0'+X_1'}} \\
 =&\ \sum_{\nu=0}^1 \Bigabs{\sum_j \idual{J_\nu x_\nu(j)}{(J_\nu^{-1})'\phi_j}{E_\nu,E_\nu'}} \\
 \le&\ \sum_{\nu=0}^1 \bignorm{( 2^{(\nu-\theta)j} J_\nu x_\nu(j))_j}_{E_\nu(\ell^q)} 
         \bignorm{( 2^{(\nu-\theta)(-j)} (J_\nu^{-1})' \phi_j))_j}_{E_\nu'(\ell^{q'})}.
\end{align*}
Taking infima we see by \cite[Thm. 3.1]{Ku:lq} that $\phi$ is in the
dual space of $(\cX_0,\cX_1)_{\theta,\ell^q}$ and with corresponding norm
$\lesssim \nn{\phi}{(\cX_0',\cX_1')_{\theta, \ell^{q'}}}$.

\medskip\noindent Conversely, let
$\phi\in\big((\cX_0,\cX_1)_{\theta,\ell^q}\big)'$. Consider the linear
subspace $Y$ of all sequences $(y_j^0,y_j^1)_{j\in\Z}$ in
$E_0(\ell^q)\times E_1(\ell^q)$ such that there exists $x\in X_0\cap X_1$
with
\begin{equation*} 
\forall j\in\Z \qquad  2^{j\theta}J_0^{-1}y_j^0+2^{(\theta-1)j}J_1^{-1}y_j^1=x 
\end{equation*}
Define the linear functional $\phi_Y$ on $Y$ by
$\phi_Y((y_j^0,y_j^1)_j):=\phi(x)$. Observe that $\phi_Y$ is
well-defined and linear. In addition, we have
\[
  \begin{aligned}
    |\phi_Y((y_j^0,y_j^1)_j)|
    = & \; |\phi(x)|\le\nn{\phi}{((\cX_0,\cX_1)_{\theta,\ell^q})'} \nn{x}{(\cX_0,\cX_1)_{\theta,\ell^q}} \\
 \le & \; \nn{\phi}{((\cX_0,\cX_1)_{\theta,\ell^q})'} \left(\sum_{\nu=0}^1\bignorm{(y_j^\nu)_j}_{E_\nu(\ell^q)}\right).
  \end{aligned}
\]
By Hahn-Banach, $\phi_Y$ has a norm-preserving extension to a linear
functional $\psi$ on $E_0(\ell^q)\times E_1(\ell^q)$ which has the form
\[
 \psi((y_j^0,y_j^1)_j)=\sum_{j\in\Z} \Big(\idual{y_j^0}{\psi_j^0}{E_0,E_0'}+\idual{y_j^1}{\psi_j^1}{E_1,E_1'}\Big),
\]
where $(\psi_j^0,\psi_j^1)_{j\in\Z}\in E'_0(\ell^{q'})\times E'_1(\ell^{q'})$ thus satisfy
\begin{align}\label{eq:psi-normest}
 \sum_{\nu=0}^1 \bignorm { (\psi_j^\nu)_j}_{E_\nu'(\ell^{q'})} \lesssim \nn{\phi}{((\cX_0,\cX_1)_{\theta,\ell^q})'}.
\end{align}
We claim that 
\begin{align}\label{eq:psik=}
 J_0'\psi_k^0=2^k J_1'\psi_k^1,\quad \mbox{$k\in\Z$}.
\end{align}
In order to see this, fix $k\in\Z$, take $x_k^0\in X_0\cap X_1$ and
let $x_j^0:=0$ for all $j\in\Z\sm\{k\}$. For $j\in\Z$, define
$x_j^1\in X_0\cap X_1$ by $2^{j\theta}x_j^0+2^{(\theta-1)j}x_j^1=0$
and put $y_j^0:=J_0x_j^0$, $y_j^1:=J_1x_j^1$. Observe that
$x_k^1=-2^kx_k^0$. Then $(y_j^0,y_j^1)_j\in Y$ and
\begin{align*}
 0=&\ \phi(0)=\psi((y_j^0,y_j^1)_j)=\sum_{\nu=0}^1\idual{y_k^\nu}{\psi_k^\nu}{E_\nu,E_\nu'}
 =\sum_{\nu=0}^1\idual{x_k^\nu}{J_\nu'\psi_k^\nu}{X_\nu,X_\nu'} \\
 =&\ \idual{x_k^0}{J_0'\psi_k^0}{X_0,X_0'}-\idual{x_k^0}{2^kJ_1'\psi_k^1}{X_1,X_1'}
 =\idual{x_k^0}{J_0'\psi_k^0-2^kJ_1'\psi_k^1}{X_0\cap X_1,X_0'+X_1'}
\end{align*}
As $x_k^0\in X_0\cap X_1$ was arbitrary, this proves \eqref{eq:psik=}
in $X_0'+X_1'$ which by $J_\nu' \psi_k^\nu\in X_\nu'$ in turn implies that
$J_\nu' \psi_k^\nu\in X_0'\cap X_1'$. We thus can define
\begin{align}\label{eq:phik}
 \phi_k:=2^{-\theta k}J_0'\psi_k^0=2^{(1-\theta)k}J_1'\psi_k^1\in X_0'\cap X_1',\quad k\in\Z.
\end{align}
Again, we take an arbitrary $x\in X_0\cap X_1$ and decompose as
$x=x_j^0+x_j^1$ with $x_j^\nu\in X_\nu$
and \eqref{eq:lq-Kdec}. Then
\begin{align*}
 \dual{x}{\phi}=&\ \phi_Y((2^{-j\theta}J_0x_j^0,2^{(1-\theta)j}J_1x_j^1)_j)
 =\sum_{j\in\Z}\left(\dual{2^{-j\theta}J_0x_j^0}{\psi_j^0}+\dual{2^{(1-\theta)j}J_1x_j^1}{\psi_j^1}\right) \\
 =&\sum_{j\in\Z} \dual{x_j^0+x_j^1}{\phi_j} =\sum_{j\in\Z}\dual{x}{\phi_j},
\end{align*}
and we conclude that $\phi=\sum_{j\in\Z}\phi_j$ weak$^*$ in
$X_0'+X_1'=(X_0\cap X_1)'$. But by \eqref{eq:phik} and
\eqref{eq:psi-normest} we have
\[
 \sum_{\nu=0}^1 \bignorm { (2^{(\theta-\nu) j} \phi_j)_j}_{E_\nu'(\ell^{q'})}
 =\sum_{\nu=0}^1 \bignorm { (\psi_j^\nu)_j}_{E_\nu'(\ell^{q'})} \lesssim\nn{\phi}{((\cX_0,\cX_1)_{\theta,\ell^q})'},
\]
which first implies $\phi=\sum_{j\in\Z}\phi_j$ in $X_0'+X_1'$ and then
finally shows $\phi\in (\cX_0',\cX_1')_{\theta,\ell^{q'}}$ with norm
$\lesssim\nn{\phi}{((\cX_0,\cX_1)_{\theta,\ell^q})'}$.
\end{proof}

\begin{remark}\rm
We remark that the assertion of Theorem~\ref{thm:dual-lq} can also be proved as an application of \cite[Prop~3.14]{LL-interpol}, see \cite[Example~3.15~(iv)]{LL-interpol}.
\end{remark}

Theorem~\ref{thm:dual-lq} allows us to give a description of the dual
space of generalized Triebel-Lizorkin spaces for an operator $A$ in
terms of generalized Triebel-Lizorkin spaces for the dual operator
$A'$.

\begin{prop}\label{prop:dual-TL}
Assume that $X$ satisfies Assumption~\assref{ass:BFS} and is reflexive, that $A$ is an $\cR_q$-sectorial operator in $X$ of angle $\om_A^{\ell^q}\in[0,\pi)$, where $q\in[1,\infty)$, and that $0\in\rho(A)$ holds. For each $\theta\in(0,1)$ we have
\begin{align}\label{eq:dual-TL}
 \big(X_{\theta,\ell^q}(A)\big)'=(A')_{-1}\big( (X')_{1-\theta,\ell^{q'}}(A') \big),
\end{align}
where the duality pairing is given by
$\dual{x}{x'}=\idual{x}{x'}{X,X'}$ for $x\in D(A)$ and $x'\in D(A')$.
Here, the operator $(A')_{-1}$ in $(X')_{-1}(A')$ is the unique
continuous extension $(A')_{-1}:X'\to (X')_{-1}(A')$ of
$A': D(A') \to X'$. Equivalent norms on $\big(X_{\theta,\ell^q}(A)\big)'$
are given by
\begin{align}\label{eq:dual-TL-norms}
 \Bignorm{\Big(\int_0^\infty \big| t^{\theta-1}\phi(tA')\big((A')_{-1}\big)^{-1}x'\big|^{q'}\,\frac{dt}t\Big)^{1/q'}}_{X'},
\end{align}
where $\phi\neq0$ is a function in $H^\infty_0(\sect{\om})$ for some
$\om\in(\om_A^{\ell^q},\pi)$ such that
$z\mapsto z^{\theta-1}\phi(z)\in H^\infty_0(\sect{\om})$.
\end{prop}

\begin{proof}
Let $\theta\in(0,1)$.
By \cite{Ku:lq} we have
\[
 X_{\theta,\ell^q}(A)=(\cX_0,\cX_1)_{\theta,\ell^q},
\]   
i.e. the generalized Triebel-Lizorkin space $X_{\theta,\ell^q}(A)$ can be
obtained as an $\ell^q$-interpolation space between the structured Banach
spaces $\cX_0:=(X,I_X,X)$ and $\cX_1:=(X_1(A),A,X)$. Observe that
$I:X\to X$ and $A:X_1(A)\to X$ are surjective (the latter by
$0\in\rho(A)$) and that we have
\[
 \cX_0'=(X',I_{X'},X')=:\cY_1 \quad\mbox{and}\quad \cX_1'=(X_1(A)',(A^{-1})',X'):=\cY_0.
\]
Application of Theorem~\ref{thm:dual-lq} and
\cite[Prop. 2.13~(a)]{Ku:lq} thus yields
\[
 \big(X_{\theta,\ell^q}(A)\big)'=(\cX_0',\cX_1')_{\theta,\ell^{q'}}=(\cX_1',\cX_0')_{1-\theta,\ell^{q'}}=(\cY_0,\cY_1)_{1-\theta,\ell^{q'}}.
\]
It is easy to see that, in a canonical way via the dual pairing
$\idual{\cdot}{\cdot}{X,X'}$, we have $X_1(A) ' = (D(A))' \simeq (X')_{-1}(A')$, the
first extrapolation space of $X'$ with respect to $A'$.
In this sense we then have
\[
 \cY_0=\cX_1'=\Big((X')_{-1}(A'), \big((A')_{-1}\big)^{-1}, X'\Big).
\]
Now, we can see that the couple $(\cY_0,\cY_1)$ is
$\ell^{q'}$-quasi-linearizable (see \cite[Def.~4.1]{Ku:lq} for this
notion). Indeed, this is similar to \cite[Example~4.3]{Ku:lq}: We let
\[
 V_\nu(t):=(t(A')_{-1})^{1-\nu}(1+t(A')_{-1})^{-1},\quad t>0,\,\nu=0,1. 
\]
Then $V_0(t)+V_1(t)=I$ on $(X')_{-1}(A')$ and we have to check that $\{t^{\nu-\rho}V_\nu(t):t>0\}$ are $\ell^{q'}$-bounded $\cY_\rho\to \cY_\nu$ for $\nu,\rho=0,1$ in the sense of \cite[Def.~2.4]{Ku:lq}. In our situation this means to show that the four sets
\begin{align}\label{eq:4sets}
 \{t^{\nu-\rho}J_\nu V_\nu(t) (J_\rho)^{-1}:t>0\},\quad \nu,\rho=0,1,
\end{align}
are $\ell^{q'}$-bounded $X'\to X'$ in the sense of Section~\ref{sec:Tr-L}, where for $\nu=0,1$, $J_\nu$ denotes the embedding onto $X'$ in $\cY_\nu$, i.e. $J_0=\big((A')_{-1}\big)^{-1}$ and $J_1=I_{X'}$, or $J_\nu=\big((A')_{-1}\big)^{\nu-1}$. We thus have, for $\nu,\rho=0,1$ and $t>0$,
\begin{align*}
 t^{\nu-\rho}J_\nu V_\nu(t) (J_\rho)^{-1}=&\ 
 t^{\nu-\rho} \big((A')_{-1}\big)^{\nu-1} (t(A')_{-1})^{1-\nu}(1+t(A')_{-1})^{-1} \big((A')_{-1}\big)^{\rho-1} \\
 =&\ (tA')^{1-\rho}(1+tA')^{-1}.
\end{align*}
Since $\cR_q$-sectoriality of $A$ in $X$ implies $\cR_{q'}$-sectoriality of $A'$ in $X'$ of angle $\om_A^{\ell^q}$, the four sets in \eqref{eq:4sets} are $\ell^{q'}$-bounded $X'\to X'$. 

Now we apply \cite[Prop.~4.2]{Ku:lq} and see that $(X_{\theta,\ell^q}(A))'=(\cY_0,\cY_1)_{1-\theta,\ell^{q'}}$ consists of all
$x'\in (X')_{-1}(A')$, for which the following norm is finite, and this expression gives an equivalent norm on $(X_{\theta,\ell^q}(A))'$:
\begin{align*}
 &\ \sum_{\nu=0,1} \Bignorm{\Big(\sum_{j\in\Z} |2^{(\nu-(1-\theta))j} J_\nu V_\nu(2^j)x' |^q\Big)^{1/q} }_{X'} \\
 =&\ \sum_{\nu=0,1} \Bignorm{\Big(\sum_{j\in\Z} \big|2^{(\nu-(1-\theta))j} \big((A')_{-1}\big)^{\nu-1} (2^j(A')_{-1})^{1-\nu} \big(1+2^j(A')_{-1}\big)^{-1} x' \big|^q\Big)^{1/q} }_{X'} \\
=&\ \sum_{\nu=0,1} \Bignorm{\Big(\sum_{j\in\Z} \big| 2^{\theta j} \big(1+2^j(A')_{-1}\big)^{-1} x' \big|^q\Big)^{1/q} }_{X'} \\
=&\ 2 \Bignorm{\Big(\sum_{j\in\Z} \big| 2^{-(1-\theta) j} (2^jA') \big(1+2^j A' \big)^{-1} \big((A')_{-1}\big)^{-1} x' \big|^q\Big)^{1/q} }_{X'}.
\end{align*}
This also proves \eqref{eq:dual-TL}. Finally, by \cite[Sec.~5]{Ku:lq} we conclude that the above norm is equivalent to the expressions in \eqref{eq:dual-TL-norms}.
\end{proof}

\section{The complex method and operators with BIP}\label{app:complex}

\subsection*{Complex interpolation revisited} We recall the definition
of the complex interpolation method, originally introduced by
Calder\'on in \cite{calderon}. Here, we follow the presentation in
\cite[Section~1.9]{Triebel}. We denote
$\mathbb{S}:=\{z\in \C:0<\Rep z< 1\}$.

Let $(X_0,X_1)$ be an interpolation couple and $\ga\in\R$. Then
$F(X_0,X_1,\ga)$ is the vector space of all functions
$f:\ov{\mathbb{S}}\to X_0+X_1$ that are continuous on
$\ov{\mathbb{S}}$, holomorphic in $\mathbb{S}$, satisfy
\[
  \sup \bigl\{ e^{-|\ga| \, |\Imp z|}  \,  \|f(z)\|_{X_0+X_1} \;: \quad z\in\ov{\mathbb{S}} \bigr\} \; < \;\infty
\]
and are such that $\R\to X_j$,
$\si\mapsto f(j+i\si)$, are continuous with
\[
  \sup \bigl\{ e^{-\ga|\si|} \nn{f(j+i\si)}{X_j} : \si\in\R \}    \bigr\} \; < \; \infty, \qquad j=0,1.
\]
By \cite[1.9.1 Theorem (a)]{Triebel}, $F(X_0,X_1,\ga)$ is a Banach space for the norm 
\begin{align*}
 \nn{f}{F(X_0,X_1,\ga)}:=\max_{j=0,1}\Big(\sup_{\si\in\R}e^{-\ga|\si|}\nn{f(j+i\si)}{X_j}\Big).
\end{align*}
For $\theta\in(0,1)$, the complex interpolation space $[X_0,X_1]_\theta$ is defined by 
\begin{align*}
 [X_0,X_1]_\theta:=\{ f(\theta) \in X_0+X_1: f\in F(X_0,X_1,\ga)\}
\end{align*}
and is equipped with the quotient norm
\begin{align*}
\nn{x}{[X_0,X_1]_\theta}=\inf\{\nn{f}{F(X_0,X_1,\ga)}:f\in F(X_0,X_1,\ga), f(\theta)=x\}.
\end{align*} 
The definition of $[X_0,X_1]_\theta$ does not depend on $\ga\in\R$,
and the norms for different $\ga$ are equivalent (see \cite[1.9.2
Theorem]{Triebel}).

\medskip\noindent Since we work with Littlewood-Paley expressions in this paper we need
a corresponding description for the complex
method.
If $\psi:(0,\infty)\to X_0\cap X_1$ is continuous 
such that we have 
\begin{align}\label{eq:C-j}
 \nn{t^j\psi(t)}{X_j}\le C_j\min\{t^\eps,t^{-\eps}\},\quad t>0,\quad\mbox{for $j=0,1$ and some $\eps>0$}. 
\end{align}
Then 
\begin{align}\label{eq:C-01}
 \nn{\psi(t)}{X_0+X_1}\le C_\psi\min\{ t^\eps, t^{-1-\eps}\},\quad t>0,\quad\mbox{for the same $\eps>0$,}
\end{align}
and we can define the Mellin transform $\mathscr{M}\psi$ via
\begin{align*}
 \mathscr{M}\psi(s):=\int_0^\infty t^{s}\psi(t)\,\frac{dt}t,\quad s\in\ov{S},
\end{align*}
where the integral is absolutely convergent in $X_0+X_1$. Observe that
$\mathscr{M}\psi$ is actually analytic on the strip
$\mathbb{S}+B(0,\eps)$ with values in $X_0+X_1$ and that we have
\begin{align*}
 \nn{\mathscr{M}\psi(s)}{X_0+X_1}\le C_\psi \int_0^\infty \min\{t^{\theta+\eps}, t^{\theta-1-\eps}\}\,\frac{dt}t\le C_\psi\int_0^\infty\min\{t^\eps,t^{-\eps}\}\,\frac{dt}t,\quad s=\theta+i\si\in\ov{\mathbb{S}}.
\end{align*}
Moreover, the functions $\R\to X_j$,
$\si\mapsto \mathscr{M}\ph(j+i\si)$, are continuous for $j=0,1$.  For
the following we use, for continuous functions $f:(0,\infty)\to X_j$
with $\nn{f(t)}{X_j}\le C\min\{t^\eps,t^{-\eps}\}$, and $\ga\in\R$ the
notation
\begin{align*}
 \nn{f}{\mathscr{F}_\ga(X_j)}:=\sup_{\si\in\R}e^{-\ga|\si|}\, \Bignorm{\int_0^\infty t^{i\si} f(t)\,\frac{dt}t}_{X_j}.
\end{align*}
where the integral is understood in the Bochner sense in $X_j$. For
$\ga=0$ we simply write $\mathscr{F}(X):=\mathscr{F}_0(X)$. Finally,
we denote by $\mathscr{C}_0(X_0,X_1)$ the set of all continuous
functions $\psi:(0,\infty)\to X_0\cap X_1$ satisfying \eqref{eq:C-j}
and
\[
 \max_{j=0,1}\nn{t\mapsto t^j\psi(t)}{\mathscr{F}(X_j)}<\infty.
\]
Then we have the following.

\begin{prop}\label{prop:cplx-interpol-norm}
Let $\theta\in(0,1)$. Then $[X_0,X_1]_\theta$ is the set of all $x\in X_0+X_1$ such that there exists a continuous function $\ph:(0,\infty)\to X_0\cap X_1$ such that $\psi: t\mapsto t^{-\theta}\ph(t)\in \mathscr{C}_0(X_0,X_1)$ and $\int_0^\infty \ph(t)\,\frac{dt}t=x$ where the integral is a Bochner integral in $X_0+X_1$. Taking the infimum of
\begin{align*}
 \max_{j=0,1} \nn{t\mapsto t^{j-\theta}\ph(t)}{\mathscr{F}(X_j)}<\infty.          
\end{align*}
over such functions $\ph$ yields an equivalent norm on $[X_0,X_1]_\theta$. 

The same assertion holds if we restrict to the class of such $\ph$
that have a bounded holomorphic extension (with values in
$X_0\cap X_1$) to each sector $\sect{\rho}$, $\rho\in[0,\pi)$.
\end{prop}

\begin{proof}
  We use \cite[1.9.1~Theorem (b)]{Triebel} which states that the
  linear span of functions $z\mapsto e^{\del z^2}e^{\la z} x$ with
  $\del>0$, $\la\in\R$ and $x\in X_0\cap X_1$ is dense in
  $F(X_0,X_1,\ga)$ for each $\ga$. The inverse Mellin transforms 
  of these functions can be calculated explicitly, see formula~\eqref{eq:mellin-inv} 
  below. They are bounded and analytic on each sector $\sect{\rho}$ 
  and have arbitrary polynomial decay at $0$ and $\infty$. Indeed, we have
  \[
  \frac1{2\pi} \int_\R t^{-i\si} e^{-\del \si^2+\la i\si}\,d\si
  =\frac1{2\pi}\int_\R e^{i\si(\la-\ln t)} e^{-\del\si^2}\,d\si=\mathscr{F}^{-1}(e^{-\del(\cdot)^2})(\la-\ln t),
  \]
  and the inverse Fourier transform of $\si\to e^{-\del\si^2}$ decays
  faster than any exponential at $\pm\infty$.
\end{proof}

\subsection*{Mellin transforms} The Mellin transform is related to the
Fourier transform $\mathscr{F}$ or the (two-sided) Laplace transform
$\mathscr{L}$ via the substitution $t=e^x$:
\begin{align*}
  \mathscr{M}\ph(i\si)=\int_\R e^{i\si x}\ph(e^x)\,dx=\mathscr{F}(\ph\circ\exp)(-\si)=\mathscr{L}(\ph\circ\exp)(-i\si).
\end{align*}
For $\ph$ such that $\mathscr{M}\ph$ is analytic on the strip
$\{ a<\Rep s<b \}$ and nice enough, this yields the Mellin inversion formula
\begin{align}\label{eq:mellin-inv}
 \ph(t)=\frac1{2\pi i} \int_{c-i\infty}^{c+i\infty} t^{-s}\mathscr{M}\ph(s)\,ds,
\end{align}
where $c\in(a,b)$. Moreover, it is clear that the Mellin transform of
a product $\psi(t)\ph(t)$ is given by convolution of $\mathscr{M}\psi$
and $\mathscr{M}\ph$ along parallels to the imaginary axis. 
\begin{align}\label{eq:mellin-conv}
 \mathscr{M}(\ph\psi)(s)=\frac1{2\pi i}\int_{c-i\infty}^{c+i\infty}\mathscr{M}\ph(z) \, \mathscr{M}\psi(s-z)\,dz.
\end{align}
From the  definition and Plancherel's theorem, we have
\begin{align}\label{eq:dual-mellin}
  \int_0^\infty \langle\ph(t),\psi(t)\rangle\,\frac{dt}t
= \int_\R \langle\ph(e^x),\psi(e^x)\rangle \,dx
= \frac1{2\pi}\int_\R \langle\mathscr{M}\ph(i\si),\mathscr{M}\psi(-i\si)\rangle\,d\si,
\end{align}
where $\ph$ has values in $X$ and $\psi$ has values in $X'$.

\medskip\noindent We calculate some Mellin transform integrals.
\begin{lemma}\label{lem:mellin-ph}
Let $\beta>\al\ge0$ and $a\in\C\sm(-\infty,0]$ and put $\ph(t):=\frac{t^\al}{(a+t)^\beta}$, $t>0$. Then
\[
 \mathscr{M}\ph(s)=\frac{a^{s+\al-\beta}}{\Ga(\beta)}\Ga(s+\al)\Ga(\beta-\al-s)
\]
for $s\in\C$ with $\Rep s\in (-\al,\beta-\al)$. 
\end{lemma}

\begin{proof}
  First let $a=1$.  We substitute $1+t=\frac1{1-x}$,
  $dt=\frac{dx}{(1-x)^2}$, $t=\frac{x}{1-x}$ and calculate
\[
 \int_0^\infty \frac{t^{s+\al}}{(1+t)^\beta}\,\frac{dt}t 
 = \int_0^1 \Big(\frac{x}{1-x}\Big)^{s+\al-1}(1-x)^{\beta}\,\frac{dx}{(1-x)^2}
 =\int_0^1 x^{s+\al-1}(1-x)^{\beta-(s+\al)-1}\,dx,
\]
which equals $\Ga(\beta)^{-1}\Ga(s+\al)\Ga(\beta-\al-s)$. For $a>0$ we have 
\[
  \int_0^\infty \frac{t^{s+\al}}{(a+t)^\beta}\,\frac{dt}t=\int_0^\infty \frac{(at)^{s+\al}}{(a+at)^\beta}\,\frac{dt}t
  = a^{s+\al-\beta}  \int_0^\infty \frac{t^{s+\al}}{(1+t)^\beta}\,\frac{dt}t,  
\]
and the general case follows by analytic continuation.
\end{proof}

\noindent In particular we have
\begin{align*}
 \mathscr{M}\ph(s)&=a^{s-1}\Ga(s)\Ga(1-s)=a^{s-1}\frac{\pi}{\sin(\pi s)},\quad 0<\Rep s<1\quad 
 \mbox{for}\quad \ph(t)=\frac{1}{a+t},\\
 \mathscr{M}\ph(s)&=a^{s-1}\Ga(s+1)\Ga(1-s)=a^{s-1}\frac{\pi s}{\sin(\pi s)},\quad -1<\Rep s<1\quad \mbox{for}\  \ph(t)=\frac{t}{(a+t)^2},\\
 \mathscr{M}\ph(s)&=\frac{a^{s-2}}2\Ga(s+1)\Ga(2-s)=\frac{a^{s-2}}2\frac{\pi s(1-s)}{\sin(\pi s)},\quad -1<\Rep s<2\quad 
 \mbox{for}\ \ph(t)=\frac{t}{(a+t)^3}.
\end{align*}
Observe that $\mathscr{M}\ph$ is meromorphic on $\C$ here with simple poles
at $k\in\Z$, $k\in\Z\sm\{0\}$, and $k\in\Z\sm\{0,1\}$, respectively.

\subsection*{Gamma function estimates}
We are particularly interested in the case $\al=\frac14$,
$\beta=\frac12$, i.e. $\ph(t)=\frac{t^{1/4}}{(a+t)^{\einhalb}}$, $t>0$, for
which we have
$\mathscr{M}\ph(s)=\frac{a^{s-1/4}}{\sqrt{\pi}}\Ga(\frac14+s)\Ga(\frac14-s)$. We
need an estimate for $s=i\si\in i\R$ and 

\begin{lemma}\label{lem:Gamma-funktion}
  We have
\[
  |\si\Ga(i\si)|^2=\frac{\pi\si}{\sinh(\pi\si)} \qquad\mbox{and}\qquad
  \mbox{$|\Ga(\frac12+i\si)|$}^2=\frac\pi{\cosh(\pi\si)},\qquad\si\in\R,
\]
and there exists a constant $C>0$ such that
\begin{align}\label{eq:Ga-14}
 C^{-1} \frac1{1+\si^2}\,\frac{\pi\si}{\sinh(\pi\si)} \le
  \mbox{$\Ga(\frac14+i\si)\Ga(\frac14-i\si)$}=|\mbox{$\Ga(\frac14+i\si)$}|^2\le
  C \frac{\pi\si}{\sinh(\pi\si)}, \quad \si\in\R.
\end{align}
\end{lemma}

\begin{proof}
  Replacing $s$ in $\Ga(s)\Ga(1-s)=\frac\pi{\sin(\pi s)}$ by
  $s+\frac12$ yields
\[
  \mbox{$\Ga(\frac12+s)\Ga(\frac12-s)$}
  =\frac\pi{\sin(\pi s+\frac\pi2)}=\frac\pi{\cos(\pi s)},
\]
whereas using $\Ga(1-s)=-s\Ga(-s)$ yields
\[
 \Ga(s)\Ga(-s)=\frac{-\pi}{s\sin(\pi s)}.
\]
Now let $s=i\si$ and recall $\ov{\Ga(s)}=\Ga(\ov{s})$,
$\cos(is)=\cosh s$, and $i\sin(is)=-\sinh s$. 
For the proof \eqref{eq:Ga-14} we quote the following from \cite[\S 8 (10) on p. 24]{Nielsen}, the formula is attributed to Lerch there (observe that we have corrected the obvious typo): For $\tau\in(0,1)$, one has
\begin{align}\label{eq:nielsen}
 \Ga(1+\tau) \frac1{\sqrt{\tau^2+\si^2}}\, 
  \le |\Ga(\tau+i\si)|\, \left(\frac{\sinh(\pi\si)}{\pi\si}\right)^{\einhalb}
  \le \Ga(1+\tau) \frac{\sqrt{1+\si^2}}{\sqrt{\tau^2+\si^2}},\quad\si\in\R.
\end{align}
In particular, for $\tau\in(0,1)$ there exist constants $c_\tau,C_\tau>0$ such that
\begin{align}\label{eq:nielsen-cor}
 \frac{c_\tau}{\sqrt{1+\si^2}}\, \left(\frac{\pi\si}{\sinh(\pi\si)}\right)^{\einhalb}
 \le |\Ga(\tau+i\si)|\le C_\tau  \left(\frac{\pi\si}{\sinh(\pi\si)}\right)^{\einhalb},\quad\si\in\R,
\end{align}
and taking $\tau=\frac14$ proves \eqref{eq:Ga-14}.
\end{proof}

\medskip

\begin{remark}\rm
The proof of \eqref{eq:nielsen} given in \cite[\S8]{Nielsen} starts from the formula
\[
 \Ga(s)=\frac{e^{-\ga s}}{s}\,\prod_{n=1}^\infty \frac{e^{s/n}}{1+\frac{s}n},
\]
from which the real part of $\log\Ga(s)$ is calculated. 
\end{remark}

\medskip

\subsection*{Operators with BIP}
Now let $A$ be a sectorial operator of angle $\om_A\in[0,\pi)$ in $X$
with dense domain and range. Let $\ph\in H^\infty_0(\sect{\rho})$ with
$\int_0^\infty \ph(t)\,\frac{dt}t=1$ where $\rho\in(\om_A,\pi)$. Then we
have
\begin{align}
 \int_0^\infty \ph(tA)x\,\frac{dt}t=x, \quad x\in D(A)\cap R(A),
\end{align}
see, e.g., \cite[Thm 5.2.6]{fc-book}. For an arbitrary $\psi\in H^\infty_0(\sect{\rho})$
and $x\in D(A)\cap R(A)$ we therefore calculate
\begin{align}\label{eq:mellin-psiA}
 \int_0^\infty t^{i\si} \psi(tA)x\,\frac{dt}t = A^{-i\si} \int_0^\infty (tA)^{i\si}\psi(tA)x\,\frac{dt}t 
 = \mathscr{M}\psi(i\si) \, A^{-i\si}x,  
\end{align}
since, for $\psi_\si(t):=t^{i\si}\psi(t)$, we have $\int_0^\infty \psi_\si(t)\,\frac{dt}t=\mathscr{M}\psi(i\si)$ (observe 
$|\psi_\si(t)|=|\psi(t)|$ for $t>0$). This is already the basis for the following.

\begin{prop}\label{prop:BIP-sqf}
  Let $A$ be a sectorial operator of angle $\om_A\in[0,\pi)$ in $X$ with
  dense domain and range and $\ph(t):=\frac{t}{(1+t)^2}$,
  $t>0$. Then 
  \begin{aufzi}
  \item\label{item:BIP-sqf-phi}  $A$ has bounded imaginary powers (BIP) in $X$ if and only if
    there is a $\ga\ge0$ such that
    \begin{align}\label{eq:BIP-ga-x}
      \nn{t\mapsto \ph(tA)x}{\mathscr{F}_\ga(X)}\lesssim\nn{x}{X}.
    \end{align}
  \item\label{item:BIP-sqf-psi} If $A$ has BIP in $X$ with
    $\theta_A\in[0,\pi)$ and $\rho\in(\theta_A,\pi)$ then, for any
    $\psi\in H^\infty_0(\sect{\rho})\sm\{0\}$ and
    $q\in[1,\infty]$, the expression
\[
\Bignorm{\si\mapsto \int_0^\infty  t^{i\si}\psi(tA)x\,\frac{dt}t}_{L^q(\R;X)}
\]
yields an equivalent norm on $X$. Moreover, for $\theta\in(0,1)$ we
have $\dot{D}(A^\theta)=[X,\dot{D}(A)]_\theta$ and for
$\tilde{\psi}\neq0$ with
$z\mapsto z^{-\theta}\tilde\psi(z)\in H^\infty_0(\sect{\rho})$, an
equivalent norm on the complex interpolation space
$[X,\dot{D}(A)]_\theta$ is given by
\[
\bignorm{t\mapsto t^{-\theta}\tilde\psi(tA)x}_{\mathscr{F}(X)}.
\]
\end{aufzi}
\end{prop}

\begin{proof}
\ref{item:BIP-sqf-phi} By the calculations above we have, for $x\in D(A)\cap R(A)$ and $\si\in\R$,
\[
 \int_0^\infty t^{i\si} \ph(tA)x\,\frac{dt}t=\frac{\pi i\si}{\sin(\pi i\si)}A^{-i\si}x=\frac{\pi\si}{\sinh(\pi\si)}A^{-i\si}x.
\]
If $A$ has BIP with in $X$ with $\n{A^{i\si}}\lesssim e^{\om|\si|}$
then \eqref{eq:BIP-ga-x} holds with $\ga=0$ if $\om<\pi$, and with
$\ga>\om-\pi$ if $\om\ge\pi$. Conversely, if \eqref{eq:BIP-ga-x} holds
then
\[
 \n{A^{-i\si}x}\lesssim e^{\ga|\si|}\frac{\sinh(\pi\si)}{\pi\si}\n{x}\lesssim e^{(\pi+\ga)|\si|}\n{x},
\]
and $A$ has BIP.

\ref{item:BIP-sqf-psi} We have to estimate the right hand side of
\eqref{eq:mellin-psiA}. We choose $\theta\in(\theta_A,\rho)$ and find
$M>0$ such that $\n{A^{i\si}x}\le Me^{\theta|\si|}\n{x}$ for all
$\si\in\R$ and $x\in X$.  For the Mellin transform of $\psi$ we
observe that dilation invariance of the integral and analytic
continuation yield
\[
  \mathscr{M}\psi(i\si)
  =\int_0^\infty t^{i\si} \psi(t)\,\frac{dt}t
  =\int_0^\infty (zt)^{i\si}\psi(zt)\,\frac{dt}t
\]
for any $z\in\sect{\rho}$. Now the right hand side equals
\begin{align*}
 e^{-\si\arg z}\int_0^\infty (|z|t)^{i\si} \psi(zt)\,\frac{dt}t
 =  e^{-\si\arg z}\int_0^\infty t^{i\si} \psi(\mbox{$\frac{z}{|z|}$}t)\,\frac{dt}t,
\end{align*}
where we used dilation invariance again. By assumption on $\psi$ we have, for some $C,\eps>0$,
\begin{align*}
 \left|\int_0^\infty t^{i\si} \psi(\mbox{$\frac{z}{|z|}$}t)\,\frac{dt}t\right| 
 \le C \int_0^\infty\min\{t^\eps,t^{-\eps}\}\,\frac{dt}t=:C_\psi<\infty.
\end{align*} 
Minimizing with respect to $\arg z$ we thus have
\begin{align*}
 |\mathscr{M}\psi(i\si)|\le&\  C_\psi\,e^{-\rho|\si|},\quad\si\in\R, \\
 \Bignorm{\int_0^\infty t^{i\si}\psi(tA)x\,\frac{dt}t}\le &\ C_\psi M e^{-(\rho-\theta)|\si|}\n{x},\quad\si\in\R,
\end{align*}
and conclude
\[
\Bignorm{\si\mapsto \int_0^\infty  t^{i\si}\psi(tA)x\,\frac{dt}t}_{L^q(\R;X)}\lesssim\n{x}.
\]
For the inverse estimate we first note
\begin{align}\label{eq:BIP-below}
  \n{A^{-i\si}x}\ge\frac1M\,e^{-\theta|\si|}\n{x},\quad\si\in\R, x\in X.
\end{align}
Then we find by $\psi\neq0$ a $\si_0\in\R$ with
$\mathscr{M}\psi(i\si_0)\neq0$ and by continuity of $\mathscr{M}\psi$
a bounded open interval $I$ around $\si_0$ and $\del>0$ such that
$|\mathscr{M}\psi(i\si)|\ge\del$ for $\si\in I$.  Invoking also
\eqref{eq:BIP-below} we hence have, for any $q\in[1,\infty]$,
\begin{align*}
 \Bignorm{\si\mapsto \int_0^\infty  t^{i\si}\psi(tA)x\,\frac{dt}t}_{L^q(\R;X)}
 \ge &\ \nn{\si\mapsto \mathscr{M}(i\si)A^{-i\si}x}{L^q(I;X)}  \\
 \ge &\ |I|^{1/q}\inf_{\si\in I}\left(|\mathscr{M}\psi(i\si)|\,\n{A^{-i\si}x}\right) \\
 \ge &\ \bigl(\del|I|^{1/q}\frac1M \min_{\si\in\ov{I}}e^{-\theta|\si|}\bigr) \,\n{x}.
\end{align*} 
We conclude
\[
 \n{x}\lesssim\Bignorm{\si\mapsto \int_0^\infty  t^{i\si}\psi(tA)x\,\frac{dt}t}_{L^q(\R;X)}.
\]
For the proof of the last assertion we fix $\theta\in(0,1)$ and $\tilde\psi$ as in the assumption. Then $\psi(z):=z^{-\theta}\tilde\psi(z)$ defines a function $\psi\in H^\infty_0(\sect{\rho})\sm\{0\}$. For $x\in D(A^\theta)$ we then have
\[ 
 t^{-\theta}\tilde\psi(tA)x=(tA)^{-\theta}\tilde\psi(tA)A^\theta x=\psi(tA)A^\theta x,
\] 
hence, for $x$ sufficiently nice, 
\[
\nn{t\mapsto t^{-\theta}\tilde\psi(tA)x}{\mathscr{F}(X)}
=\Bignorm{\si\mapsto \int_0^\infty t^{i\si}\psi(tA)A^\theta x\,\frac{dt}t}_{L^\infty(\R;X)}\simeq\n{A^\theta x}.
\]
We recall that $[X,\dot{D}(A)]_\theta=\dot{D}(A^\theta)$ holds by \cite[Theorem 15.28]{KuW:levico}, the homogeneous version of \cite[1.15.3 Theorem]{Triebel}. 
\end{proof}

\providecommand{\bysame}{\leavevmode\hbox to3em{\hrulefill}\thinspace}

\end{document}